\documentclass[reqno,12pt]{amsart}

\usepackage{
amsfonts,
latexsym,
amssymb,
enumerate,
verbatim,
mathrsfs,
rotate,
centernot,
color,
}

\newcommand{\labbel}[1]{\label{#1} [[{\bf #1}]]}  
\renewcommand{\labbel}{\label}

\newcommand{\arxiv}{}

\usepackage{stackengine,
graphicx}
\stackMath
\def\dashsubseteq{\mathrel{
  \stackinset{l}{0pt}{c}{}{\wrule[2pt]{5pt}{.5pt}}{
  \stackinset{l}{1.6pt}{c}{}{\wrule[-10pt]{.5pt}{3pt}}{
  \stackinset{l}{0.5pt}{c}{1.5pt}{\rotatebox[origin=center]{45}{\wrule[2pt]{.5pt}{3pt}}}{
  \stackinset{l}{0.5pt}{c}{}{\rotatebox[origin=center]{-45}{\wrule[-5pt]{.5pt}{3pt}}}{
  \stackinset{l}{3.35pt}{c}{}{\wrule[-4pt]{.5pt}{10pt}}{
  \stackinset{l}{5.1pt}{c}{}{\wrule[-4pt]{.5pt}{10pt}}{
 \subseteq}}}}}}
}}
\newcommand\wrule[3][0pt]{\textcolor{white}{\rule[#1]{#2}{#3}}}

\newtheorem{theorem}{Theorem}[section]
\newtheorem{lemma}[theorem]{Lemma}

\newtheorem{proposition}[theorem]{Proposition} 
 
\newtheorem{corollary}[theorem]{Corollary}

\newtheorem*{claim*}{Claim}

\newtheorem*{theorem*}{Theorem}
\newtheorem*{proposition*}{Proposition}
\newtheorem*{corollary*}{Corollary}
\newtheorem*{lemma*}{Lemma}
\newtheorem*{scholion*}{Scholion}

\theoremstyle{definition}
\newtheorem{definition}[theorem]{Definition}

\theoremstyle{remark}
\newtheorem{remark}[theorem]{Remark}

\newtheorem*{remark*}{Remark}
\newtheorem*{remarks*}{Remarks}
 
\newtheorem{example}[theorem]{Example}
\newtheorem{examples}[theorem]{Examples}

\newtheorem*{observation*}{Observation}

 \allowdisplaybreaks[1]

\numberwithin{equation}{section}

\begin{document}

\title{Preservation of  superamalgamation by expansions}

\author{Paolo Lipparini} 
\address{Dipartimento di Supermatematica\\Viale della  Ricerca
 Scientifica\\Universit\`a di Roma ``Tor Vergata'' 
\\I-00133 ROME ITALY}

\email{lipparin@axp.mat.uniroma2.it}

\subjclass{03C52;  06F99;  06A15;  03G25}

\keywords{Closure operation; isotone operation;
antitone operation;  amalgamation property;
superamalgamation; semilattice; Fra\"\i ss\'e limit; complete lattice;
Boolean algebra; decidable theory}

\thanks{Work performed under the auspices of G.N.S.A.G.A. Work 
partially supported by PRIN 2012 ``Logica, Modelli e Insiemi''.
The author acknowledges the MIUR Department Project awarded to the
Department of Mathematics, University of Rome Tor Vergata, CUP
E83C18000100006.}

\begin{abstract}
The superamalgamation property is a strong form of 
the amalgamation property which applies
 to ordered structures; it has found many applications
in algebraic logic. 
We show that 
superamalgamation has some interest also from
the pure model-theoretical point of view.
Under a completion assumption, we prove that the superamalgamation property
for some class of ordered structures  implies  strong amalgamation 
for classes with added  
 operations, including isotone, idempotent, extensive,
antitone and closure operations. 

Thus, for example, partially ordered sets, 
semilattices, lattices, Boolean algebras and Heyting algebras
with an isotone extensive operation
 (or an operation as above) have the strong amalgamation property.
The theory of join semilattices with a closure operation has 
model completion.  The set of  universal  
 consequences
of the theory of Boolean algebras (or posets,
semilattices, distributive lattices)
 with a closure or isotone, etc.,
 operation is decidable.
\end{abstract}

\maketitle

\section{Introduction} \labbel{intro}

The amalgamation property (AP) is a classical tool
in algebra \cite{KMPT} 
and has found many applications in logic,  particularly, 
in model theory \cite{H} and algebraic logic \cite{CP,GM,GG,MMT}. 

B. J{\'o}nsson \cite{J1} asked whether there are 
 \emph{general  results  that  assert  that  if  an  elementary  
class  is  characterized 
by  axioms of such and  such a form,  then  this class has  the  amalgamation 
property}.
J{\'o}nsson's problem, as stated, seems to have a positive solution
only in an incredibly small number of cases.
On the other hand, there are many situations in which
theories which are already known 
to have the amalgamation property can be combined  \cite{GG,apu}
 or modified in order to produce many other theories with 
AP. Classical examples include
 adding operators, e.~g.,  \cite{lop,MS,W,Z}.

Here we present a  general 
result
in the above wake.
If some class of ordered structures
has the superamalgamation property
and satisfies an appropriate completion property,
then we can add operations with
a number of properties (isotone,
antitone,  extensive, idempotent, closure\dots) 
in such a way that  amalgamation is preserved.
 The added operations possibly depend on many arguments, e. g., a
$4$-ary operation which is isotone on the first two components and antitone on the
last two components.
The assumptions apply to the classes of
  Boolean
algebras, Heyting algebras,  lattices, join semilattices,
 meet semilattices and partially ordered sets (henceforth, \emph{posets}, for short). 
 See Theorem \ref{superap} and
Corollary \ref{corsuperap} below.    
Recall that the superamalgamation property is a natural strengthening of 
AP for ordered structures, and has found significant applications
in algebraic logic \cite{GM,KH}. 
Our results show that 
superamalgamation has some interest also from
the pure model-theoretical point of view, even when
dealing with AP alone.

As  applications,  we show that  the theory of join semilattices
 with a closure operation has 
model completion  (Corollary \ref{smk})    
and we generally get the existence
of Fra\"\i ss\'e limits for the subclass of finite members
of the classes under consideration  (Corollary \ref{corsuperap}(3)).  
As a consequence of our main Extension Lemma \ref{exte}, 
we prove that the 
set of  universal  
 consequences
of the theory of Boolean algebras 
 (or posets, 
semilattices, distributive lattices)  
with  a finite number of operations of the kind taken into consideration
  is decidable. 
 See Corollary \ref{badecid}.   

The paper is divided as follows. 
Section \ref{prel}  is devoted to preliminaries 
 and auxiliary results.    
In Section \ref{extsec}
we present the Extension Lemma \ref{exte}, which broadly generalizes former
results. Given a partial function $G$ on some complete lattice,
we devise the exact conditions under which
$G$ can be extended to a total function satisfying one of
the properties taken into account.
Classical embedding results for posets then
imply extension results for various kinds of posets with
an operation, for example, semilattices or lattices
with a closure operation. 
 See Corollary \ref{plap}.  

In Section \ref{sapimp}
we prove our more general result  
  about amalgamation, Theorem \ref{superap}:  
if $\mathcal S$ is a class of partially ordered
structures, $\mathcal S$ has the superamalgamation property and
every member of $\mathcal S$ can be embedded into a complete
 member of $\mathcal S$,
then superamalgamation is preserved by adding a closure operation,
or an isotone (extensive, idempotent, etc.) operation.  
\emph{Complete} is always meant in the lattice theoretical sense.
In particular, the result applies when $\mathcal S$ is one of the
following classes: posets,
semilattices,  lattices, Boolean algebras, Heyting algebras
 (Corollary \ref{corsuperap}(1)).    
We can also add simultaneously many operations at a time,
possibly with some comparability conditions
 (Corollary \ref{corsuperap}(2)).    
As a consequence, we get the existence of Fra\"\i ss\'e limits
in the case of a finite language
 (Corollary \ref{corsuperap}(3)).    
 Since the theory of  join semilattices with 
a closure operation is locally finite,  we get the existence of
a model completion
 (Corollary \ref{smk}).

 Section \ref{lfsec} exploits the power
 of the extension 
 lemma from Section \ref{extsec}
and does not rely on the amalgamation property.
 In Theorem \ref{decid}    
we consider  a universal locally finite theory $T$  
with an order relation such that every finite model of $T$ 
is lattice-ordered  (or, more generally,
can be extended to a finite lattice-ordered model). 
If $T^+$  extends $T$ in a finite language with further operations and
with axioms asserting that the operations  satisfy one of the conditions taken 
into account, then the set of  universal  
  consequences of $T^+$  is decidable
(this would be obvious for $T$, but is far from being obvious for 
$T^+$, which is not necessarily locally finite). This applies to 
Boolean algebras 
 (or posets,
semilattices, distributive lattices)   
with a finite number of further operations.
 See Corollary \ref{badecid}.  
Section \ref{fur} presents a few further remarks; in particular,
we show that, in some special cases, the completeness assumptions 
in Sections \ref{extsec} -  \ref{sapimp}
 can be relaxed to some extent; nevertheless, counterexamples
are provided showing that some degree of completeness
is necessary.

\section{Preliminaries} \labbel{prel}

We shall consider the following properties of 
a unary operation $K$  on a partially ordered set, \emph{poset}, for short: 
\begin{align} \labbel{extd} 
& a \leq Ka && \text{ \emph{extensive}, }
\\
 \labbel{contrad} 
& Ka \leq a && \text{ \emph{contractive}, }
\\ 
\labbel{isod}
& a \leq b \text{ implies } Ka \leq Kb  && \text{  \emph{isotone}, }
\\ 
\labbel{antidef}
& a \leq b \text{ implies } Kb \leq Ka  && \text{  \emph{antitone}, }
\\
\labbel{ided}
&  KKa=Ka && \text{ \emph{idempotent}, }
\\
\labbel{invod}
&  KKa=a && \text{ \emph{involutive}. }
\end{align} 
The conditions are supposed to hold for all the elements $a,b$ of the poset 
$\mathbf P$  under consideration.
 We frequently consider a partial operation $G$  defined 
on some subset $D$ of $P$ and we say that 
$G$ is extensive or contractive if 
\eqref{extd} or \eqref{contrad} holds for all 
$a \in D$. 

 We shall also deal with $n$-ary  operations.
An $n$-ary  operation is \emph{isotone on the $i^{\text{th}} $
component} if it is isotone as a unary operation when
the argument of the $i^{\text{th}} $ component varies and
the other arguments are kept fixed. Operations 
 \emph{antitone on the $i^{\text{th}} $ component} are
defined correspondingly.

A \emph{closure} (\emph{interior}) operation
is an extensive (contractive), isotone and idempotent operation.  
In the presence
of a semilattice operation, some authors include 
an additivity or multiplicativity requirement
in the definitions of a closure and an interior operation.
We shall adopt the more general convention
\cite{E,Ha} according to which no additivity
or multiplicativity assumption is made,
unless explicitly mentioned otherwise.

We shall generally deal with
\emph{ordered structures}, namely, posets
with possibly additional
operations or relations.
For definiteness, such structures will be considered as \emph{models}
in the sense of classical model theory \cite{H}, but
 in Section \ref{sapimp}      
there are more general possibilities,
e.~g., topological or infinitary structures.
The precise setting will generally not be relevant for our purposes,
insofar
as the meaning of type (or signature, or language) and of embedding 
are clear. When some poset is, say, a lattice,
 we shall explicitly mention whether
we are considering   order-embeddings,
or the stronger notion of lattice-embeddings,
that is, embeddings preserving the lattice operations. 
Lattice operations shall be denoted by 
 $\vee$ and $\wedge$.    
Their infinitary extensions are indicated by 
$\sum$ and  $\prod$.

For the sake of simplicity, in the following
definitions classes of structures
shall be always supposed to be closed under isomorphism.
If $\mathbf A$ and $\mathbf B$ 
are structures of the same type,
with base sets, respectively, $A$ and $B$,
then $ \mathbf A $ 
is said to be a \emph{substructure} of 
$   \mathbf B$, in symbols,
$ \mathbf A \subseteq  \mathbf B$,
if $A \subseteq B$ as sets and the
inclusion is an embedding of 
$ \mathbf A $ into $   \mathbf B$.

\begin{definition} \labbel{sap}   
 A class $\mathcal  S$  of structures  of the same type
and closed under isomorphism
has the \emph{strong amalgamation property} (SAP)
 if the following holds. Whenever 
$\mathbf A, \mathbf B, \mathbf C \in \mathcal  S$,
 $ \mathbf C \subseteq  \mathbf A$,
 $ \mathbf C \subseteq  \mathbf B$
and $C=A \cap B$,
then there is a structure
$\mathbf D \in \mathcal  S$ such that 
$ \mathbf A \subseteq  \mathbf D$ and
$ \mathbf B \subseteq  \mathbf D$.
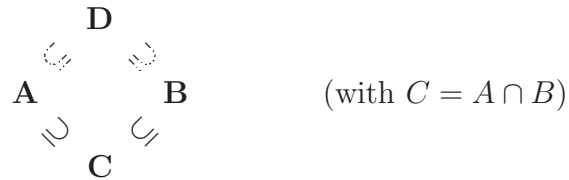
\begin{equation*}
\phantom{\qquad \qquad  \text{ (with $C=A \cap B$)}  }
 \begin{matrix} 
 \mathbf D \cr
$\rotatebox[origin=c]{45}{$ \dashsubseteq $}$
 \ \quad\  
$\reflectbox {\rotatebox[origin=c]{45}{$ \dashsubseteq $}}$ \cr
\mathbf A \quad \quad \quad \quad \mathbf B \cr
$\rotatebox[origin=c]{-45}{$ \supseteq $}$ \ \quad\ 
$\rotatebox[origin=c]{45}{$ \subseteq $}$
 \cr
 \mathbf C
 \end{matrix}   
\qquad \qquad  \text{ (with $C=A \cap B$)}  
\end{equation*}

 In the case of ordered structures,
the \emph{superamalgamation property}
asserts that, under the above assumptions, 
there exists some $\mathbf D$ as above 
with the additional property that, 
for every  $a \in A \setminus B $ and $b \in B \setminus A $, 
  \begin{enumerate}[(a)]
\item 
if $a \leq _{\mathbf D}  b $,  
then there is $c \in C$ such that  $a \leq_{\mathbf A} c $,
$  c \leq _{\mathbf B} b $ and, symmetrically,
\item 
if $b \leq _{\mathbf D}  a $,  
then there is $c \in C$ such that  $b \leq_{\mathbf B} c $,
$  c \leq _{\mathbf A} a $.
  \end{enumerate} 

In particular,  every class with the superamalgamation property 
has the strong amalgamation property.

A theory $T$ has the strong amalgamation property (the superamalgamation property)
if the class of models of $T$ has such a property.
\end{definition} 

In principle, our results apply to structures
with many partial orders at a time.
In such a situation, we always suppose that some specific partial order
is selected; the superamalgamation property 
is always meant to refer to such partial order.

We refer to \cite{E,GM,Ha,H,KMPT,MS,MMT} 
for more information about the above notions.

 The following result is proved by
standard arguments in fixed point theory; see, e.~g., Chapter 12 in \cite{R},
in particular, Theorem 12.9 therein.
We present the explicit proof  since
 the construction
shall be used in the course of some proofs in
the next section.
 If $K$ and $J$ are two operations
defined on the same poset $\mathbf P$, we say that  $K$ is 
(pointwise) \emph{larger} than
$J$, or that $J$ is \emph{smaller} than $K$, if $Kx \geq Jx$, 
for every $x \in P$. Here ``larger''  is always intended in the broader sense
of ``larger than or equal to''.

\begin{lemma} \labbel{HH}
Suppose that $\mathbf P$ is a poset and
$H$ is an isotone operation
on $\mathbf P$ such that $HHx \leq Hx$, for every $x \in P$.
If every nonempty infinite
chain in the range of $H$ has a meet in $\mathbf P$
(in particular, if every nonempty infinite
chain in $\mathbf P$ has a meet),
then there is the largest operation $K$ 
  among those isotone and idempotent operations in $\mathbf P$
which are smaller than $H$. 
 \end{lemma} 

\begin{proof} 
For $\alpha$ a nonzero ordinal, let
\begin{equation}\labbel{ord}      
\begin{aligned} 
K^1x&=Hx,
\\
  K^{ \alpha +1} x &= H K^ \alpha x \quad \text{ and } 
\\ 
 K^ \beta x &= \prod _{0 < \alpha  < \beta }
 K ^\alpha x \quad \text{ for $\beta$ limit.}
\end{aligned}
 \end{equation}

We need to justify the limit case in  definition \eqref{ord}.
For this it is enough to prove by transfinite 
induction on $\beta >0$ that, for every $x \in P$,  the sequence  
$(K ^ \delta x) _{0 < \delta  \leq \beta} $ is defined and decreasing
(not necessarily in the strict sense).

The base case $\beta=1$ is immediate.  

If $\beta$ is a limit ordinal, then the inductive assumption implies that 
the sequence $(K ^ \delta x) _{0 < \delta  < \beta} $ is
defined and decreasing, so that 
$ \{ \, K ^ \delta x  \mid   0 < \delta  < \beta\,\} $
is a chain, hence $K^ \beta x $ is defined, by the assumption on $\mathbf P$
(notice that, if $\delta$ is limit, then $K ^ \delta x $ 
does not necessarily belong to the range of $H$; however, since the sequence
is decreasing and $\beta$ is a limit ordinal,
the meet of  $ \{ \, K ^ \delta x  \mid   0 < \delta  < \beta, \delta \text{ not limit}\,\} $
exists if and only if the meet of $ \{ \, K ^ \delta x  \mid   0 < \delta  < \beta\,\} $ exists,
and if they exist, they are equal).  
Moreover, $K^ \beta  \leq K^ \delta x $, for $\delta < \beta $, by construction. 

It remains to prove the induction
step, that is,  $K^{ \beta  + 1}  x \leq K^ \beta  x$, for every 
 ordinal $\beta>1$ and $x \in P$,
assuming the inductive hypothesis, that is, 
 $(K ^ \delta x) _{0 < \delta  \leq \beta} $
defined and decreasing. 
When $\beta>1$  is a successor ordinal
the induction step is immediate from
the assumption that $HHx \leq Hx$, for every $x \in P$.
Indeed, in this case  $K^ \beta  x = HK^ \gamma  x $,
where $\gamma$ is the predecessor of $ \beta $,
thus $K^{ \beta + 1}  x  = H K^ \beta x = HHK^ \gamma  x
 \leq HK^ \gamma  x = K^ \beta x$, by the assumption on $H$
with $K^ \gamma  x$ in place of $x$.
We now show that $K^{ \beta  + 1}  x \leq K^ \beta  x$ 
also when
$\beta$ is a limit ordinal.
For every $\gamma$ with $0 <\gamma < \beta $ 
we have $K^{ \beta+1} x = H K^ \beta x
 \leq H K^ \gamma  x
  = K^{ \gamma +1} x \leq K^ \gamma x$,
by the definition of $K^ \beta$, since $H$ is assumed to be isotone
and since $K^ \beta x \leq K^ \gamma  x$.
The  inequality
$K^{ \gamma +1} x \leq K^ \gamma x$ 
 follows from the inductive assumption.
 We have showed that
$K^{ \beta+1} x \leq  K^ \gamma x$,
for every nonzero $\gamma < \beta $
and this means exactly 
$K^{ \beta+1} x \leq  K^ \beta x$.

We have proved that the sequence $K^ \alpha x$, $\alpha$ 
a nonzero ordinal,
is decreasing, so that it eventually stabilizes. 
This justifies the next definition.
For every $x \in P$ we  set 
\begin{align}   
\labbel{kkkkk} 
&K x= K ^ \alpha x
 \end{align}
where $\alpha$ is the smallest nonzero ordinal such that 
$K^{ \alpha +1} x=  K^ \alpha x$. 

From \eqref{kkkkk} it follows that
\begin{equation}\labbel{kuk}      
HKx=HK^{ \alpha } x = K^{ \alpha +1} x= K^{ \alpha } x= K x .
  \end{equation}
  We  now show that $K$ is the desired operation.
Indeed, $K$ is  idempotent, since if
$K^{ \alpha +1} x=  K^ \alpha x$, then 
$K^2Kx=HHK^{\alpha}x=HK^{\alpha+1}x 
=HK^{\alpha}x=K^1Kx$, thus
$KKx=K^1Kx=HKx=Kx$, by \eqref{kuk}.
Notice that, on the other hand, we are not assuming 
$Hx \leq x$.
Moreover, $K$ is smaller than $H$, since, for every $x \in P$, 
$(K ^ \alpha x) _{0 < \alpha } $ is
a decreasing sequence, thus, for $\alpha$ as in \eqref{kkkkk},
$K  x = K ^ \alpha x \leq K ^ 1 x = Hx$.   
We now check that $K$ is isotone. Notice that, for every ordinal $ \beta >0 $, 
each $K^ \beta $  is  isotone, by induction on $\beta$ 
and using the assumption that $H$ is isotone.  
If $x \in P$ and $K^{ \alpha +1} x=  K^ \alpha x$,
then $K^{ \beta} x=  K^ \alpha x$,
for every $\beta \geq \alpha$. 
Thus if $x \leq y \in P$, then, for some suitably large
ordinal $\beta$, $K x=  K^ \beta  x \leq  K^ \beta  y = Ky  $.
 Hence $K$ is isotone. Finally, if $J$
is an isotone and idempotent operation
smaller than  $H$,
we check by induction on $\delta >0$ that
$Jx \leq   K^ \delta x$,
for every $x \in P$ and every nonzero ordinal $\delta$.
It follows  that $J$ is smaller than  $K$.
The base case $\delta=1$ is exactly the assumption that
$J$ is smaller than  $H$, while the limit case
is immediate from the inductive hypothesis.
To prove the successor step,
notice that     
if $Jx \leq K^ \delta x$, then
$Jx =JJx \leq JK^ \delta x \leq
HK^ \delta x = K ^{ \delta +1}x $,
by idempotency and isotony of $J$, 
and again by the assumption that
$J$ is smaller than  $H$.

Thus $K$ is the largest
isotone idempotent operation
among those smaller than $H$.  
\end{proof}    

The assumption $HHx \leq Hx$ is necessary in   Lemma \ref{HH}:
see Remark \ref{HHrem}(a) below.
It is necessary to iterate $H$ in the proof of Lemma \ref{HH}:
see Remark \ref{HHrem}(c) below.

\begin{proposition}  \labbel{minb2}
Suppose that 
$\mathbf P$  is a bounded  complete
lattice.

If $ \{ \,  K_i \mid  i \in I \,\} $
is a  family of isotone and idempotent operations
 on $P$,
then there is the largest isotone idempotent  operation $K$ 
which is smaller than each $K_i$. 
 
In other words, the set of all the isotone and idempotent operations
 on $P$ is a complete bounded lattice,  under the ordering $\leq_p$
given by $J \leq_p K$ if $J$ is pointwise smaller than $K$.  
\end{proposition}  

\begin{proof}
For $x \in P$,  define $Hx= \prod _{i \in I} K_i x $.

For every $i \in I$, we have
  $HHx \leq K_i Hx
\leq K_i K_i x=K_i x$, 
since  $Hx \leq K_i x$,
for every $x \in P$ (in particular, we can
take $Hx$ in place of $x$) and then  using
isotony and idempotency of $K_i$.
We have showed that 
  $HHx \leq K_i x$,
for every $i \in I$, 
hence 
  $HHx \leq \prod _{i \in I} K_i x = Hx$.

Since each $K_i$ is isotone, then $H$ is isotone, 
hence we can apply Lemma \ref{HH} in order to get 
 the largest operation $K$ 
  among those isotone and idempotent operations
which are smaller than $H$.
Now the conclusion follows from the fact that,
by the definition of $H$, some
operation $K$ is smaller than $H$ 
if and only if $K$ is smaller than all the $K_i$s. 
 \end{proof}    

The next theorem  is a
collection of  folklore results, but some of them can be hardly
found explicitly mentioned in the literature.

\begin{theorem} \labbel{josup}
The classes of partially ordered sets,  meet semilattices,  join semilattices, 
lattices,  Boolean algebras and Heyting algebras
all have the superamalgamation property.
The same applies to the classes of finite such structures.
 \end{theorem}

 \begin{proof} 
 Though not explicitly stated, 
the proof of Lemma 3.3 in \cite{J} provides
the superamalgamation property for 
posets.
Full details are presented, for example, in \cite[Proposition 2.1]{duerel}. 
For short, if $\mathbf A$, $\mathbf  B$ and $\mathbf  C$ 
are posets to be amalgamated, then 
$\mathbf  D= (A \cup B, \leq _{\mathbf D})$
is a superamalgamating structure, where $\leq _{\mathbf D}$
 is the smallest transitive
relation containing ${\leq _{\mathbf A}} \cup {\leq _{\mathbf B}} $.
The key argument in  \cite[Lemma 3.3]{J} amounts to show that such a relation equals 
 ${\leq _{\mathbf A}} \cup {\leq _{\mathbf B}} \cup 
({\leq _{\mathbf A}} \circ {\leq _{\mathbf B}} ) \cup \allowbreak 
({\leq _{\mathbf B}} \circ {\leq _{\mathbf A}} )$, where, say,
$(a,b) \in {\leq _{\mathbf A}} \circ {\leq _{\mathbf B}}$ means that
there is $c \in A \cap B =C$ such that 
$ a \leq _{\mathbf A} c $ and $ c \leq _{\mathbf B} b$.  
 
An argument in the proof\footnote{ The 
arguments in the proof of
\cite[Theorem 3.5]{J} essentially amount
to the proof of the strong amalgamation property
(condition IV in the terminology from \cite{J}) 
for the class of lattices; all the rest is a consequence of the former 
Theorem 2.11 in \cite{J}.  In particular,
\cite[Theorem 3.5]{J} is stated under the assumption of
 the Generalized Continuum Hypothesis,
but the assumption is not used directly
 in most arguments in the proof,
it is used only when relying to 2.11.   
}
 of \cite[Theorems 3.5]{J}
on page 205 then  goes on by showing that if $\mathbf  C$ is a lattice, then
all existing meets and joins in $\mathbf A$ are preserved in $\mathbf  D$,
as constructed above, and similarly for existing meets and joins in $\mathbf  B$.
In particular,
if $\mathbf A$, $\mathbf  B$ and
 $\mathbf  C$ are lattices, then meets and joins of 
$\mathbf A$ and  $\mathbf  B$ are preserved
in the poset $\mathbf  D$. This does not mean that 
$\mathbf  D$ is a lattice, since, say, the meet of some 
$a \in A$ with some $b \in B$ might not exist in $\mathbf  D$. 
However, any poset can be embedded into some lattice
in such a way that existing meets and joins are preserved,
thus $\mathbf  D$ can be order-embedded into some lattice
$\mathbf L$ in such a way that $\mathbf A$ and $\mathbf  B$ 
are lattice-embedded into $\mathbf L$
(the argument can be reformulated in terms of \emph{partial lattices},
see \cite[p. 455]{G}).
This proves strong amalgamation for lattices and,
as noticed in the second paragraph of 
 \cite{Fl}, the same argument works for semilattices.
Since the definition of superamalgamation deals only with the order structure,
the fact that $\mathbf  D$ superamalgamates 
$\mathbf A$ and $\mathbf  B$ over $\mathbf  C$ as posets  implies that  
$\mathbf  L$ superamalgamates 
$\mathbf A$ and $\mathbf  B$ over $\mathbf  C$ as lattices,
respectively, semilattices.

Thus we know that the classes of posets, 
 lattices and semilattices all have the superamalgamation property.

Boolean algebras and Heyting algebras are well-known to have the
superamalgamation property; 
 see, e.~g.,  \cite{GM}.
In detail, \cite[Theorem 6.1]{GM}, among other, proves 
that a variety $V$ of Heyting algebras has the amalgamation property if and only if
$V$ has the superamalgamation property. Notice that Heyting
algebras are generally called \emph{pseudoboolean algebras} in \cite{GM}.
Theorem 6.1 in \cite{GM} is stated in terms of superintuitionistic logics;
however, there is a bijective correspondence between superintuitionistic logics and 
varieties of Heyting algebras; see \cite[p. 172]{GM}.  
Then \cite[Proposition 6.8]{GM} goes on by showing
that the variety of all Heyting algebras has the amalgamation property.
In \cite{GM} the amalgamation property for Boolean algebras
is obtained again from Theorem 6.1 therein; see \cite[p. 177]{GM}.

Superamalgamation  for Boolean algebras
 can also be obtained more directly as a consequence of \cite[Corollary 11.22]{Ko}.
There it is shown that if $\mathbf  D$ is the free product of two
Boolean algebras $\mathbf A$ and $\mathbf  B$ with amalgamated
subalgebra $\mathbf  C$, then, for every 
 $a \in A$ and  $b \in B$ such that  $ab=0$ in $\mathbf  D$,
there is $c \in C$ such that $a \leq c$ and $b \leq -c$.  
Taking $-b$ in place of $b$, $a(-b)=0$ means $a \leq b$ and then
the mentioned corollary gives $-b \leq -c$, that is, $c \leq b$.
  
The last statement follows from the fact that all the
above constructions preserve finiteness \arxiv{(this needs to be checked:
for example, let $T$ be the theory in the pure language
of equality asserting that if there are more than two objects,
then there are infinitely many objects. Then $T$ has the strong amalgamation property,
but the class of the finite models of $T$ has the
amalgamation property but not
the strong amalgamation property).}
\end{proof}

\section{An extension lemma} \labbel{extsec} 

The next lemma is a broad generalization of earlier
results, e.~g.,   \cite[Lemma 2.3]{MT}, \cite[Theorems 1 - 4]{Se},
\cite[Proposition 3]{ecca}, 
 \cite[Theorem 3.19]{E},   
 \cite[Lemma 8.1]{Sc}.
The present treatment has several advantages over
 earlier versions:
first, no assumption is made on the subset $D$,
second, extensiveness is not necessarily assumed,
though its presence  simplifies statements
and  proofs.
Finally, we do not need 
to work with complete Boolean algebras.
In the simpler case  it is enough to have just an arbitrary poset.
In the general case a complete bounded lattice is enough and no 
form of distributivity is necessary.
\arxiv{On the other hand, distributivity seems to be  necessary 
 when  dealing with additive operations, as we plan to show 
in a further work. } 

The general form of the next lemma, in particular, the
fact that no special assumption is imposed on $D$, will prove particularly
useful in Section \ref{lfsec} below, where it will lead to some
decidability results.

 In the next lemma we shall consider a poset $\mathbf P$,
a subset 
$D $ of $  P$ and a function $G: D \to P$.
We shall deal with the following properties of $G$.   
\begin{align}
\labbel{ide}
&\text{for all $a \in D$, if } Ga \in D, \text{ then } GGa=Ga,
\\
\labbel{inv}
&\text{for all $a \in D$, if } Ga \in D, \text{ then } GGa=a, 
\\
\labbel{invb}
&  \text{for all $a, b \in D$, if } Ga =Gb, \text{ then } a=b,    
\\
\labbel{iso}
& \text{for all $a, b \in D$, if } a \leq b \text{ then } Ga \leq Gb,
\\
\labbel{isoide}     
& \text{for all $a, b \in D$, if } a \leq Gb \text{ then } Ga \leq Gb, 
\\
\labbel{isoidebis}  
& \text{for all $a, b \in D$, if } Ga \leq b \text{ then } Ga \leq Gb,
\\
\labbel{ant}
& \text{for all $a, b \in D$, if } a \leq b \text{ then } Gb \leq Ga.
\end{align}

\begin{lemma} \labbel{exte}
 Suppose that $\mathbf P$  is a 
partially ordered set,
$D \subseteq P$ and $G: D \to P$ is a function.  Then,
for each line in the following table,
$G$ can be extended to an operation $K$ on $ P$
satisfying the properties listed in the second column
if and only if $G$ satisfies the properties listed in the third column

\emph{\begin{equation*}   
\begin{array}{|c|c|c|}
\hline&&
\\[-10pt]
  \text{} 
& \textit{$G$ can be extended to $K$ such that} 
  & \textit{if and only if}  
\\
\hline&&
\\[-10pt]
  \text{A1e} 
& \text{$K$ is extensive} 
  & \text{$G$ is extensive}  
\\
\hline&&
\\[-10pt]
\text{A1c} 
&  \text{$K$ is contractive}   
& \text{$G$ is contractive}  
\\
\hline&&
\\[-10pt]
  \text{A2} 
& \text{$K$ is idempotent} 
& \text{$G$ satisfies \eqref{ide}}  
\\
\hline&&
\\[-10pt]
  \text{A2e} 
& \begin{array}{c}
\text{$K$ is idempotent}
\\
\text{and extensive}
\end{array}  
& \begin{array}{c}
\text{$G$ satisfies \eqref{ide}}
\\
\text{and is extensive}
\end{array}   
\\
\hline&&
\\[-10pt]
  \text{A2c} 
& \begin{array}{c}
\text{$K$ is idempotent}
\\
\text{and contractive}
\end{array}  
& \begin{array}{c}
\text{$G$ satisfies \eqref{ide}}
\\
\text{and is contractive}
\end{array}   
\\
\hline&&
\\[-10pt]
  \text{A3} 
& \text{$K$ is an involution} 
& \text{\ \ $G$ satisfies \eqref{inv} and \eqref{invb}\ \ }  
\\
\hline
\end{array}   
\end{equation*}
}
and, under the further assumption that $\mathbf P$  is a bounded  complete
lattice,
\emph{\begin{equation*}   
\begin{array}{|c|c|c|}
\hline&&
\\[-10pt]
  \text{} 
 & \begin{array}{c}
\textit{$G$ can be extended to $K$}
\\
\textit{such that}
\end{array}   
  & \textit{if and only if}  
\\
\hline&&
\\[-10pt]
  \text{B1} 
& \text{$K$ is isotone} 
& \text{$G$ satisfies \eqref{iso}}  
\\
\hline&&
\\[-10pt]
  \text{B1e} 
& \text{$K$ is isotone and extensive } 
& \text{$G$ satisfies \eqref{iso} and is extensive}  
\\
\hline&&
\\[-10pt]
  \text{B1c} 
& \text{$K$ is isotone and contractive} 
& \text{$G$ satisfies \eqref{iso} and is contractive}  
\\
\hline&&
\\[-10pt]
  \text{B2} 
& \text{$K$ is isotone and idempotent} 
& \text{$G$ satisfies \eqref{iso}, \eqref{isoide} and \eqref{isoidebis}}  
\\
\hline&&
\\[-10pt]
  \text{B3} 
& \text{$K$ is a closure operation} 
& \text{$G$ satisfies \eqref{isoide} and  is extensive}  
\\
\hline&&
\\[-10pt]
  \text{B4} 
& \text{$K$ is an interior operation} 
& \text{$G$ satisfies \eqref{isoidebis} and  is contractive}  
\\
\hline&&
\\[-10pt]
  \text{B5} 
& \text{$K$ is antitone} 
& \text{$G$ satisfies \eqref{ant}}  
\\
\hline
\end{array}   
\end{equation*}
}

Suppose further that $n \in \mathbb N \setminus \{ 0 \} $,
$i,j \in \mathbb N  $, $i+j \leq n$, 
$X \subseteq P^n$ and
 $V: X \to P$  is a function. 
Then $V$ can be extended to an $n$-ary operation $F$ on $ P$
such that
\begin{enumerate}[({C}1)] 
\item 
$F$ is isotone 
on the first $i$ components and
antitone on the last $j$ components  
 \quad if and only if,    \quad
 for all $n$-uples  $ \bar a = \{ a_1, \dots, a_n \}, 
 \bar b = \{ b_1, \dots, b_n \} \allowbreak \in X$,
\begin{equation}\labbel{ison}
\begin{aligned}
&a_h \leq b_h, \text{ for every $h \leq i$, }
a_h = b_h, \text{ for every $h $ with $i < h \leq n-j$,}
\\
& \text{  and } a_h \geq b_h, \text{ for every $h > n- j$ } 
\qquad \text{ imply }  \qquad  V \bar a \leq V \bar b.
\end{aligned}
\end{equation}
\item
If $i\geq 1$, in (C1) we can additionally obtain
  $x_1 \leq F\bar x$ for all $ \bar{x} $, 
 provided \eqref{ison} holds and 
$a_1 \leq V\bar a$, for all 
 $ \bar a \in X$.

Similarly for any subset of  the first $i$ 
 variables.

\item
More generally, 
given an $i$-ary  lattice term $t$, 
in (C1) we can additionally obtain   $t(x_1, \dots, x_i) \allowbreak  \leq F\bar x$,
 provided \eqref{ison} holds and
$t(a_1, \dots, \allowbreak a_i) \allowbreak \leq V\bar a$,
 for every $n$-uple  $ \bar a \in X$.
  \end{enumerate}

In cases (B1) - (C3) the assumption that $\mathbf P$  is a bounded  complete
lattice can be weakened to ``every subset of the range of $G$ (resp., of 
$V$) has a meet in $\mathbf P$'', except for cases
 (B1c) and  (B4), in which cases  the assumption
can be weakened to ``every subset of the range of $G$ has a join in $\mathbf P$'',
 and for case (C3) in which, by definiteness, we must
further assume that $\mathbf P$ is 
 a lattice.  
 \end{lemma}

\begin{proof}
In each case, necessity is 
 immediate, without any assumptions on
the completeness of $\mathbf P$, 
assuming that $\mathbf P$ is a lattice in case (C3). 
\arxiv{Notice also that the order structure on $P$ is not
relevant in case (A3).} 
 We now prove sufficiency.

 (A1e) - (A2c)   
 In each case define $K:P \to P$ by 
\begin{align} 
 \labbel{kide}   
 \qquad \qquad Kx&=Gx &&\text{ if $x \in D $,  \qquad \qquad }
\\
\labbel{kidee}   
Kx&=x &&\text{ if $x \notin D $}. 
\end{align} 
Clause \eqref{kide}  assures that 
$K$ extends $G$. Moreover, $K$ 
is extensive or contractive if and only if so is 
$G$. Assume  \eqref{ide}. If $a \in D$ and $Ka \in D $,
then $KKa=Ka$  because of \eqref{kide} and \eqref{ide}.
If $a \in D$ and $Ka \notin D $,
then $KKa=Ka$  because of \eqref{kidee},
which implies $KKa=Ka$ also 
when $a \notin D$. 
Hence $K$ extends $G$ and is idempotent.

 (A3)
 Define $K:P \to P$ by 
\begin{align} 
 \labbel{A}   
 Kx&=Gx &&\text{ if $x \in D $,}
\\
\labbel{B}   
Kx&=a &&\text{ if $Ga=x $, for some $a \in D$,} 
\\
\labbel{C}   
Kx&=x &&\text{ otherwise, that is, neither $x \in D$, nor $x \in Im\, G$}. 
\end{align} 

Conditions \eqref{inv} and  \eqref{invb}
imply that the outcomes of $K$ agree in all the overlapping cases. 
Clause \eqref{A} implies that $K$ extends $G$ and $K$ 
is an involution by construction.   

(B1) Define, for each $x \in P$,
\begin{align}
\labbel{kiso} \tag{Case B1}
 &&  Kx= \prod \{ \,  Gb  \mid b\in D, x \leq b  \,\}. 
\end{align} 

By definition, $K$ is isotone.
It remains to show that $K$ extends $G$.
So let  $x \in D$. We can take $b=x$  
in the set defining $Kx$, hence $Kx \leq Gx$.
On the other hand, for every $b \in D$ 
  such that  $x \leq b$, we have 
$Gx \leq Gb$, by  \eqref{iso}, hence $Gx 
\leq  \prod \{ \,  Gb  \mid b\in D, x \leq b  \,\} = Kx$.
Thus $Kx = Gx$.

If $G$ is extensive, then $K$  is 
 extensive.
If $G$  is contractive, then $K$, as defined,
need not be contractive. 
However, we can get a contractive $K$ 
by performing the dual construction, namely, defining
$Kx= \sum \{ \,  Gb  \mid b\in D, b \leq x  \,\}$.

(B2) 
For $x \in P$ let 
\begin{equation} 
\labbel{kisoide1}
   Hx = \prod \{ \,  Gb  \mid b\in D \text{ and either }  x \leq b \text{ or }  x \leq Gb  \,\}
 \end{equation}   

 By definition, $H$ is isotone.
We now check that 
\begin{equation}\labbel{hh}    
\text{$HHx \leq Hx$,
for every $x \in P$.}    \end{equation}
Indeed, if $Gb$ belongs to the defining set for $Hx$ in \eqref{kisoide1},
then   $Hx \leq Gb$, by the very definition of $Hx$; thus
if in \eqref{kisoide1} we consider
$Hx$ in place of $x$, then 
$HHx \leq Gb$, since $Hx \leq Gb$.
We have showed that $HHx$ is $\leq$ than every factor in the set 
whose meet gives $Hx$, hence it follows that $HHx \leq Hx$.   

Let $K$ be the operation defined
in \eqref{kkkkk} in the proof of Lemma \ref{HH},
with respect to $H$ defined in \eqref{kisoide1} above:
recall that $K$ is obtained by iterating $H$.
By Lemma \ref{HH}, $K$ is isotone and idempotent, hence it remains  
to check that  if $x \in D$, then $Kx =Gx$.
We first prove that if $x \in D$, then $Hx =Gx$.    
Indeed,  suppose that $x \in D$. If $x \leq b \in D$, then $Gx \leq Gb$, by \eqref{iso}.
If   $x \leq Gb$ and $b \in D$, then $Gx \leq Gb$, by \eqref{isoide}.
Thus $Gx \leq Hx = \prod \{ \,  Gb  \mid b\in D \text{ and either }
 x \leq b \text{ or }  x \leq Gb  \,\}$. Also  $Gx \geq H x $,
since we can take $b=x$ in 
the defining set in \eqref{kisoide1}. Hence $ H x = Gx$. 

 We now prove that if $x \in D$, then $HHx =Hx$.
By \eqref{hh}, $HHx \leq  H x$.  
On the other hand, 
using the already proved  fact that, for $x \in D$,  
$Gx= H x $, we get that if  $b \in D$ and    
either $H x \leq b$ or $H x \leq Gb$,  
 then  $H x  = Gx  \leq Gb$  in each case,
applying \eqref{isoidebis} in the former eventuality. 
Hence $H x \leq 
 \prod \{ \,  Gb  \mid b \in D \text{ and either } H x \leq b \text{ or } H x \leq Gb  \,\}
=HH x $.

We have proved that
if $x \in D$, then 
$ Gx =H x = HH x  $.
 According to the definition in \eqref{ord}
in the proof of Lemma \ref{HH},  this means that 
$Gx =  K^1 x = K^2x $, thus
$ Kx= K^1 x =Gx $, by the definition of $K$ in \eqref{kkkkk}.

(B3) follows  from 
 the proof of 
(B2),
since if $G$ is extensive, then 
$H$, as defined by \eqref{kisoide1}, 
turns out to be extensive, as well,
hence
 $K$ given by 
 \eqref{kkkkk} is extensive.  
Since condition \eqref{isoidebis}
is  satisfied by an extensive $G$ and, moreover,
\eqref{isoide} implies  \eqref{iso} when $G$
is extensive, 
then we get (B3).

Let us mention, however, that a direct 
proof of (B3)  is much simpler, just define
\begin{align}
\labbel{gkclo} \tag{Case B3}
 &&  Kx= \prod \{ \,  Gb  \mid b\in D, x \leq Gb  \,\}
\end{align} 

The operation $K$ is isotone and extensive.
If $b \in D$  and $x \leq Gb$, then $x \leq Kx \leq Gb$.
Thus,
 for every $b \in D$,     
$ x \leq Gb $ if and only if $  Kx \leq Gb $,
that is $KKx = Kx$, namely,
$K$ is idempotent. 

Because of  \eqref{isoide},
if $x \in D$, then $Gx \leq Kx$.
Conversely, since $G$ is assumed to be extensive, 
we can take $b=x$ in the defining set for $Kx$,
getting  
$Kx \leq Gx$.
Thus  $Kx  =   
 Gx$,
for $x \in D$, that is, 
$K$ extends $G$. 
Compare \cite[Lemma 2.3]{MT},
 \cite[Theorem 3.19]{E}   
and \cite[Lemma 8.1]{Sc}.

In comparison with 
\eqref{kisoide1}, notice that 
if $G$ is extensive  and $b \in D$,    
then $a \leq b$  
implies $a \leq Gb$. This means that in the extensive
case $H$ 
 from  \eqref{kisoide1} and $K$ 
from \eqref{gkclo}    
 coincide and, more importantly,
 no transfinite iteration is needed.

(B4) is dual to (B3).
In detail, we can choose 
\begin{align}
\labbel{gkinte} \tag{Case B4}
 &&  Kx= \sum \{ \,  
 G   
b  \mid b\in D, 
 G   
b \leq x  \,\} 
\end{align} 
 in order to get an interior operation extending $G$.

(B5) Define
\begin{align}
\labbel{kant} \tag{Case B5}
 &&  Kx= \prod \{ \,  Gb  \mid b\in D, b \leq x  \,\}. 
\end{align} 

The operation $K$ is antitone.
Arguing as in case (B1) we see that $K$ extends $G$.
Indeed, if $x \in D$, then, taking $b=x$  
in the set defining $Kx$, we get  $Kx \leq Gx$.
If $b \in D$ 
  and   $b \leq x$, we have 
$Gx \leq Gb$, by  \eqref{ant}, hence $Gx \leq  Kx$.

(C1) It is enough to define 
\begin{equation}
\labbel{kison} \tag{Case C1}
\begin{aligned}   
 & F\bar x = \prod \{ \, V \bar  b  \mid
 \bar b\in X, x_h \leq b_h, 
\text{ for  $h \leq i$, }
 x_h = b_h, \text{ for $h $ with} 
\\ & \phantom { F\bar x = \prod \{ \, } 
i < h \leq n-j, \text{ and }
 x_h \geq b_h, \text{ for } h > n- j\,\}. 
 \end{aligned}
\end{equation} 
 
(C2) follows from the above definition in
\eqref{kison}. 

More generally (C3) follows from the fact that lattice
 terms are isotone on each component.
Indeed, fix $\bar x \in P ^{  n } $ and $X^* \subseteq X$. 
If
$x_1 \leq b_1$, \dots, $x_i \leq b_i$ and 
$t(b_1, \dots, b_i) \leq V \bar b $,
for every $ \bar{b} \in X^*$,
then
$t(x_1, \dots, x_i) \leq t(b_1, \dots, b_i) \leq V \bar b $,
for every $ \bar{b} \in X^*$.
Hence 
$t(x_1, \dots, x_i) \leq 
 \prod \{ \, V \bar  b  \mid \bar b\in X^* \,\}$. 
 
The last statement follows from the proof, since we have always
taken meets (or joins) of  subsets of the range of $G$ or $V$. 
\end{proof}    

 Some completeness assumptions are
necessary in Lemma \ref{exte} in cases (B1) - (C3);
see Remark \ref{basta}(b) - (e).   

By a \emph{comparability condition}  
between two, say, unary operations
$H$ and $K$ on the same set $X$, we mean a condition 
of the form $Hx \leq Kx$ for every $x \in X$.
Namely, a comparability condition is
always a $\leq$-condition, we shall not deal
with $<$-conditions here.

 Lemma \ref{exte} holds when applied
simultaneously to two or more partial functions,
and comparability conditions can be preserved.
This is the content of  Lemma \ref{compa} below. In most cases the result
is immediate from the proof of \ref{exte}; however, 
some details need to be
worked out in cases (A3), (B2) - (B4).

\begin{lemma} \labbel{compa}
Assume the hypotheses  in Lemma \ref{exte}.
  \begin{enumerate}   
 \item  
For each item (A1e) - (B5), assume that 
$G_{{\circ}}: D \to P$ and $G_{{\bullet}}: D \to P$
are functions satisfying the corresponding sufficient condition.
 In case (A3) assume further that
$G_{{\bullet}}b \in D$ 
and $G_{{\circ}}b \in D$, for every $b \in D$.  

Under the above assumptions, if $G_{{\circ}}b \leq G_{{\bullet}}b$,
for every $b \in D$, then there are operations
$K_{{\circ}} $ and $  K_{{\bullet}}$
defined on the whole of $P$, 
extending  respectively $G_{{\circ}}$ and $G_{{\bullet}}$,
satisfying the 
corresponding condition in the middle column
in the tables in Lemma \ref{exte} 
 and such that 
$K_{{\circ}}x \leq K_{{\bullet}}x$,
for every $x \in P$.  
\item
For some fixed $n$ and some $X \subseteq P^n$, assume that  
$V_{{\circ}}: X \to P$ and $V_{{\bullet}}: X \to P$ are functions
satisfying the condition in Lemma \ref{exte}(C), and let
$F_{{\circ}}$ and $F_{{\bullet}}$ be 
defined as in the proof of \ref{exte}.

If $V_{{\circ}} \bar b \leq V_{{\bullet}} \bar b$,
for every $\bar b \in X$, then 
$F_{{\circ}}\bar x \leq F_{{\bullet}}\bar x$,
for every $\bar x \in P^n$.
  \end{enumerate}
 \end{lemma} 

\begin{proof}
 In cases (A1e) - (B1c), (B5)
  let $K_{{\circ}}$ and $K_{{\bullet}}$ be correspondingly
defined as in the proof of Lemma \ref{exte}.
The conclusion is elementary in cases (A1e) - (A2c)  
and follows from the definition of $K$ in cases
(B1) - (B1c) and (B5).
Say, in case (B1), we have
 $K_{{\circ}}x= \prod \{ \,  G_{{\circ}}b  \mid b\in D, x \leq b  \,\}$ 
and 
$K_{{\bullet}}x= \prod \{ \,  G_{{\bullet}}b  \mid b\in D, x \leq b  \,\}$.
Since  $G_{{\circ}}b \leq G_{{\bullet}}b$,
for every $b \in D$,  then every element in the set 
whose meet gives $K_{{\bullet}}$ 
is bounded below by some 
element in the set 
whose meet gives $K_{{\circ}}$
and this implies  $K_{{\circ}}x \leq K_{{\bullet}}x$.
The argument is similar in case (B5), as well 
 as in part (2) of the present lemma. For case (A3) just notice that,
under the additional assumption, condition \eqref{B}
becomes redundant, while comparability is clearly
preserved by clauses \eqref{A} and \eqref{C}. 

We now deal with case (B3) and
shall postpone the more involved case (B2).
Let  $K_{{\circ}}$ and $K_{{\bullet}}$ be 
given by \eqref{gkclo} in the proof of  Lemma \ref{exte}.
It is not necessarily the case 
that $K_{{\circ}} x \leq K_{{\bullet}}x$,
for all $x \in P$: see Remark  \ref{compa2}(b) below.
Hence we need consider another operation
in place of   $K_{{\circ}}$.
Let $K^*_{{\circ}}x= K_{{\circ}}x \wedge K_{{\bullet}}x$,
for $x \in P$.
It is elementary to see that $K^*_{{\circ}} $
is a closure operation, since both $ K_{{\circ}} $ and $ K_{{\bullet}}$
are.
The argument is well-known; the only nontrivial part is idempotence.
Let $x \in P$ and $y=K^*_{{\circ}}x$. Since $x \leq y \leq K_{{\circ}}x $,
then   $K_{{\circ}}x \leq K_{{\circ}}y \leq K_{{\circ}}K_{{\circ}}x= K_{{\circ}}x $,
hence $K_{{\circ}}y = K_{{\circ}}x$.
Similarly,  $K_{{\bullet}}y=K_{{\bullet}}x$,
thus  $K^*_{{\circ}}K^*_{{\circ}}x=K^*_{{\circ}}y= K_{{\circ}}y \wedge K_{{\bullet}}y=
K_{{\circ}}x \wedge  K_{{\bullet}}x = K^*_{{\circ}}x$,
proving idempotence of $K^*_{{\circ}}$.

Since $K_{{\circ}}$ and $K_{{\bullet}}$ extend, respectively,
$G_{{\circ}}$ and $G_{{\bullet}}$, we have
$K^*_{{\circ}}b= G_{{\circ}}b \wedge G_{{\bullet}}b
=G_{{\circ}}b$, for every $b \in D$, since, by assumption,
 $G_{{\circ}}b  \leq G_{{\bullet}}b$. Hence
 $K^*_{{\circ}}$ extends $G_{{\circ}}$.
Thus the operations $K^*_{{\circ}}$ 
(in place of $K_{{\circ}}$)  
and $K_{{\bullet}}$ witness the conclusion of the 
present lemma in case (B3).
Case (B4) is dual.

In order to prove  case (B2) we shall use Proposition \ref{minb2}. 
Let  $K_{{\circ}}$ and $K_{{\bullet}}$ be 
given by the proof of Lemma \ref{exte}. As in case (B3),
we need to use another operation  $K_{{\circ}}^*$
in place of $K_{{\circ}}$.  Let
$I= \{ {\circ}, {\bullet} \} $ in  Proposition \ref{minb2}
and let $H$ and $K$ be given by the proofs of  \ref{HH} and \ref{minb2}.
Recall from the proof of Proposition \ref{minb2} that, in this special case, 
$Hx= K_{{\circ}}x \wedge K_{{\bullet}} x$; then the proof of
Lemma \ref{HH} is applied and $K$ is obtained by transfinitely iterating $H$
until it assumes a constant value.  
By Proposition \ref{minb2}, $K$ is isotone, idempotent and
smaller than $K_{{\bullet}}$, so that we can conclude the proof
if we show that  $K$ extends 
$G_{{\circ}}$, considering $K_{{\circ}}^*=K$ in place of $K_{{\circ}}$.

From the proof of Proposition \ref{minb2} we have 
$Hx= K_{{\circ}}x \wedge K_{{\bullet}} x$, for $x \in P$. 
Let us fix  $b \in D$. By assumption,   
$G_{{\circ}}b  \leq G_{{\bullet}}b$, thus
$K_{{\circ}}b \leq K_{{\bullet}}b$ and $Hb= 
K_{{\circ}}b \wedge K_{{\bullet}} b =K_{{\circ}}b=G_{{\circ}}b$,
since $K_{{\circ}}$ and $K_{{\bullet}}$ extend, respectively,
$G_{{\circ}}$ and $G_{{\bullet}}$.

We want to show that $HHb=Hb$, so that  \eqref{ord}
and \eqref{kkkkk} give $Kb=Hb=G_{{\circ}}b$, the conclusion
we need.    We first recall the definition of  $K_{{\bullet}}$
from the proof of Lemma \ref{exte}.  In detail, 
$K_{{\bullet}}$ is obtained by iterating the  operation
introduced in \eqref{kisoide1}, recalled below with suitable relabelings: 
\begin{equation} 
\labbel{kisoidebu}
   H_{{\bullet}} x =
 \prod \{ \,  G_{{\bullet}}d  \mid 
d\in D \text{ and either }  x \leq d \text{ or }  x \leq G_{{\bullet}}d  \,\}.
 \end{equation}   

We first check that $H_{{\bullet}}G_{{\circ}}b \geq G_{{\circ}}b$.
So let $x$ be $G_{{\circ}}b$   in \eqref{kisoidebu}.
If $G_{{\bullet}}d$ belongs to the set in   \eqref{kisoidebu}
because $x= G_{{\circ}}b \leq d$, then
$ G_{{\circ}}b \leq  G_{{\circ}}d $, since 
$ G_{{\circ}}$ is assumed to satisfy \eqref{isoidebis}.
Thus $ G_{{\circ}}b \leq  G_{{\circ}}d \leq G_{{\bullet}}d $,
by the comparability assumption between $G_{{\circ}}$ and $G_{{\bullet}}$.
Otherwise,   $G_{{\bullet}}d$ belongs to the set in   \eqref{kisoidebu}
because $x= G_{{\circ}}b  \leq G_{{\bullet}}d $;
in conclusion, $H_{{\bullet}}  G_{{\circ}}b  $ is obtained
as the meet of a set of elements which are all $ \geq G_{{\circ}}b $,
hence   $H_{{\bullet}}G_{{\circ}}b \geq G_{{\circ}}b$.
Since, recalling \eqref{ord},  $K_{{\bullet}} $ is obtained by iterating transfinitely
$H_{{\bullet}} $, we get $K^ \alpha _{{\bullet}}G_{{\circ}}b \geq G_{{\circ}}b$,
for every nonzero ordinal $\alpha$,
since $H_{{\bullet}} $ is isotone, hence
$K _{{\bullet}}G_{{\circ}}b \geq G_{{\circ}}b$.

We now compute 
$ HHb = HK_{{\circ}}b  =
 K_{{\circ}} K_{{\circ}}b \wedge K_{{\bullet}}K_{{\circ}}b
=  K_{{\circ}}b \wedge K_{{\bullet}}G_{{\circ}}b =
 G_{{\circ}}b \wedge K_{{\bullet}}G_{{\circ}}b
=  G_{{\circ}}b = Hb$,
since we already know that
$Hb=G_{{\circ}}b=K_{{\circ}}b  $, then
using the definition of $H$, idempotency of
  $K_{{\circ}}$ and the just proved inequality
$K _{{\bullet}}G_{{\circ}}b \geq G_{{\circ}}b$.

Since in the above argument 
$b$ was an arbitrary element of $D$,
we get  $HHb =Hb=G_{{\circ}}b$, for every $b \in D$.
As remarked above, this means that $K$ extends $G_{{\circ}}$,
thus $K_{{\circ}}^*=K$ (in place of $K_{{\circ}}$) and   $K _{{\bullet}}$
 satisfy the desired conclusion.    
 \end{proof}

 The additional assumption is needed
in Lemma \ref{compa} 
in  case  (A3). See Remark \ref{compa2}(a) below. 

\begin{remark} \labbel{molt}
We have stated Lemma \ref{compa} for just two operations
only for simplicity: the analogue of Lemma \ref{compa} holds
for any family of operations and an arbitrary set
of comparability conditions. In detail, assume the hypotheses of 
Lemma \ref{exte} and fix some item (W) chosen from
(A1e) - (C3) in the statements there. Let
$(Z, \preccurlyeq)$ be a partially ordered set of indices,
and $(G_z) _{z \in Z}  $ be a $Z$-indexed sequence  of functions
from $D$ to $P$ such that, say in the unary case,
$G_z(b) \leq G_{z'}(b)$, for every $b \in D$ and $z \preccurlyeq z' \in Z$.
In case (A3) assume further that $G_z(b) \in D$, for every $b \in D$
and $z \in Z$.      

If each $G_z$ satisfies the  condition on the right 
in (W), then there is a way of
extending simultaneously each $G_z$ to a  total function  
$K_z  $ on $P$  in such a way that each $K_z$
satisfies the corresponding condition   and moreover
$K_z(x) \leq K_{z'}(x)$, for every $x \in P$ and $z \preccurlyeq z' \in Z$.

 This is proved as in Lemma \ref{compa} in cases (A1e) - (B1c), (B5) - (C3),
since in such cases we can always consider the functions $K_z$
or $F_z$ provided by the proof of Lemma \ref{exte}.
To prove  case (B3),
let the functions $K_z$ $(z \in Z)$ 
be given by  Lemma \ref{exte}.
 For every $x \in P$ and  $z \in Z$, define 
\begin{equation}\labbel{tutut}
   K^*_z x= \prod \{ \, K_{z'} x \mid   z' \in Z, z \preccurlyeq z' \, \}  .
   \end{equation}  
The same arguments as in the proof of Lemma \ref{compa}
show that each $K^*_z$ is a closure operator which extends $G_z$.
 The desired comparability conditions follow directly from 
\eqref{tutut}, since if  $z \preccurlyeq w \in Z$, then 
the set defining  $K^*_{w}x$ is contained in the set
defining $K^*_z x$.
Case (B4) is dual. 

In case (B2) the equation \eqref{tutut}
should be used to define some (not necessarily idempotent)
functions $H^*_z$ ($z \in Z$) which need  to be transfinitely 
iterated as in the proofs of
Lemmas \ref{HH}, \ref{compa}  and Proposition \ref{minb2},
 in order to obtain isotone and idempotent operations
 $K^*_z$ ($z \in Z$).
The above argument for case (B3) shows that
$H^*_{z}x \leq H^*_{w}x$, for every
$z \preccurlyeq w \in Z$ and $x \in P$.
Since each $H^*_{z}$ is isotone,
then, applying the definitions in \eqref{ord} and \eqref{kkkkk} 
from the proof of Lemma  \ref{HH} 
to 
both   $H^*_{z} $ and $  H^*_{w}$,
 we get 
  $K^*_{z}x \leq K^*_{w}x$.

The comparability assumptions on the $G_z$s
imply that  $H^*_{z} b = G_z b $, for every $z \in Z$ 
and $b \in D$.  
The arguments in the proof of Lemma \ref{compa} in this same
case (B2) show that $K_{w}G_z b \geq G_z b$,  for every   
$z \preccurlyeq w \in Z$ and $b \in D$ (here
the index $w$ is in place of $\bullet$ and the index $z$ 
is in place of $\circ$).
From $H^*_{z} b = G_z b $ and by the definition of $H^*_{z}$ we get 
$H^*_{z}H^*_{z} b = H^*_{z} G_z b =
 \prod \{ \, K_{w} G_z b   \mid   w \in Z, z \preccurlyeq w \, \}
\geq G_z b = H^*_{z} b $. From the proof of 
Proposition \ref{minb2} we have  
$H^*_{z} H^*_{z} b  \leq H^*_{z} b $,
hence  $H^*_{z} H^*_{z} b  = H^*_{z} b = G_z b $
and \eqref{kkkkk} gives  
 $ K^*_{z} b  = H^*_{z} b = G_z b $, thus
 $ K^*_{z} $ extends $  G_z  $, concluding the proof of 
case (B2).
 \end{remark}

\begin{definition} \labbel{list}
In view of Lemma \ref{exte},
we shall consider the following properties
of a unary operation 
in an ordered  structure: 
(A1e) extensive, 
(A1c) contractive,
(A2) idempotent,
(A2e) idempotent and extensive, 
(A2c) idempotent and  contractive,
(A3) involutive,
(B1)  isotone, 
 (B1e)  isotone and extensive, 
 (B1c)  isotone and contractive,
(B2) isotone and idempotent,
(B3) a closure operation (that is, 
 isotone, extensive and idempotent),
(B4) an interior operation (that is, 
 isotone, contractive and idempotent), 
(B5) an antitone operation, and, 
for $n$-ary operations,
(C1)$_{i,j,n}$  
isotone 
on the first $i$ components and
antitone on the last $j$ components,
and, possibly, 
(C2)$_{i,j,n,h}$  
 satisfying also   $x_1 \leq F\bar x, \dots, 
x_h \leq F\bar x$, for some $h \leq i$,  more generally,
 for lattice-ordered structures,    
(C3)$_{i,j,n,t}$ 
 satisfying    $t(x_1, \dots, x_i) \leq F\bar x$,
for some given lattice term $t$.
 \end{definition}

\begin{corollary} \labbel{plap}
Suppose that (W) is any one of the properties
(A1e) - 
 (C2)    
listed in  Definition \ref{list}.
 \begin{enumerate}   
\item 
If  $\mathbf Q$ is a  poset
with an operation satisfying (W)  
and
$\iota$ is an order-embedding of $\mathbf Q$
into some bounded complete lattice  $\mathbf P$,
then $\mathbf P$ can be expanded by adding an
operation satisfying (W)
 in such a way that $\iota$  is an embedding
with respect to the operation.

\item
Every  poset  $\mathbf Q$
with an operation satisfying (W)
can be order-embedded into  some complete bounded  lattice $\mathbf P$
with an operation satisfying (W) and
in such a way that the embedding preserves the operation
and all existing, possibly infinitary, meets and joins in 
$\mathbf Q$.

\item
Every  poset  
with an operation satisfying (W)
can be order-embedded into  some complete atomic Boolean algebra 
with an operation satisfying (W) and
in such a way that the embedding preserves the operation  and  
all existing, possibly infinitary, meets (alternatively,  joins).  

\item
Every  distributive lattice with an operation 
satisfying (W)
 can be lat\-tice-embedded into
 some complete atomic Boolean lattice with   an operation
satisfying (W) and
in such a way that the embedding also preserves the operation.

\item
In all the above cases
we can add simultaneously any number of operations,
possibly of distinct arities, and
possibly satisfying distinct properties chosen from 
(A1e) -  (C2).  
 The construction
 can be performed in such a way that it  
 preserves comparability conditions
among operations satisfying the same property.
\end{enumerate} 
 \end{corollary}

 \begin{proof} 
(1) In cases (A1e) - (B5) define $G$ on $\iota(Q)$
by $G(\iota(a))=\iota(K_{\mathbf Q}a)$,
where $K_{\mathbf Q}$ is the given operation  on $\mathbf Q$. 
This is a good definition, since $\iota$ is injective.
The respective conditions 
among \eqref{ide} - \eqref{ant}  
in Lemma \ref{exte} are  satisfied by $G$ on 
$D= \iota(Q)$, since, by assumption,
they are satisfied by $K_{\mathbf Q}$ 
in $\mathbf Q$  and $\iota$ is an order-embedding.
By 
Lemma \ref{exte}, 
$G$ can be extended on the whole of
$P$ to an operation satisfying the desired property.
With respect to this operation $\iota$ turns out to be an embedding
by the very definition of $G$. 
\arxiv{
Of course, in cases (A1e) - (A3) we do not need 
the completeness assumption on $\mathbf P$.} 

In cases (C1) -  (C2)    
 define
$V$ on $(\iota(Q))^n$
by $V(\iota(a_1), \dots,\iota(a_n))=
\iota(F_{\mathbf Q}(a_1, \allowbreak  \dots, \allowbreak a_n))$
and argue similarly.

(2)  follows from (1), since
every poset can be embedded into  some bounded complete  lattice
in such a way that existing meets and joins  are preserved
\cite[ Ch. 1,  
 Theorems 10.6, 10.7]{Ha}.

(3) - (4) follow similarly by known
results about embeddings of ordered structures, e.~g.,
\cite[ Ch. 1,  
 Theorems 9.9, 9.10]{Ha} and
\cite[Theorem 153]{G}.

(5) The constructions  of the lattice 
$\mathbf P$  in (2) and of  the Boolean  lattices in 
(3) - (4) do not depend on the operation, hence we can add 
as many operations as we want at the same time.
The last statement follows from Lemma \ref{compa}
 and Remark \ref{molt}.  
\end{proof}  

 Item (2) (under the further assumption that $\mathbf Q$ 
is a lattice) and item (4) in Corollary \ref{plap} apply
also in case (C3). 
Clause (2) for
a closure operation 
appears in \cite[Corollary 3.20]{E}. 
Clause (3) for meet-semilattices with
a closure operation appears in  \cite[Proposition 3.2 and 
Lemma 3.4]{Ja}.
  
\arxiv{Notice that the case of join-semilattices is not the 
dual case, since the dual of a  meet-semilattice with a closure operation
is a  join-semilattice with an interior operation. 
In detail, if the semilattice operation is written multiplicatively,
extensiveness is equivalent to $x \cdot Kx = Kx$ in join 
semilattices; to  $x \cdot Kx = x$, instead, 
in meet semilattices.

A \emph{reduct} of some structure is a structure
in which some operations or relations are forgotten.
A \emph{subreduct} is a substructure of some reduct.  
It follows from Corollary \ref{plap}(3) that, say,
if $\mathcal  S$ is the class of all 
 Boolean algebras with a closure operator,
then the class of all subreducts of members of $\mathcal  S$  
to the language of posets with an operator is the class
of posets with a closure operator. 
As another example, if we consider lattice operations
as ternary relations $\vee(x,y,z)$,  $\wedge(x,y,z)$
and  $\mathcal H$ is the class of all 
lattices (in the above relational sense) with a closure operator,
then the class of all substructures of members of $\mathcal H$
is the class of partial lattices with a closure operator: use (2). 
Similar consequences can be obtained in all the other cases.} 

\section{Superamalgamation implies 
amalgamation for expanded structures} \labbel{sapimp}

 As mentioned in the 
section on preliminaries, we shall deal with
models in the classical model-theoretical sense.
Classes of models are always meant to be of the same type
and closed under isomorphism.
The proof of the next Theorem \ref{superap}
can be applied without essential modifications to
a somewhat broader setting, for example, dealing with
topological or infinitary structures. 
An even more general version in a categorical setting is
possible; however, details become quite cumbersome and we 
know no significant application; hence we shall provide
details (elsewhere)  if and when some applications are found. 

 We require embeddings to be at 
least order-embeddings. If we expand  $\mathcal S$ to some class
$\mathcal S_1$ by adding one or more operations, an
embedding for $\mathcal S_1$ is meant to be an embedding
for $\mathcal S$ which in addition respects the new operations,
say, $\iota(Kx)= K \iota(x)$, for unary operations. 
 When dealing with case (C3) we shall always assume
that structures are lattice-ordered and that
embeddings are at least lattice-embeddings. 

The assumption that 
$\mathcal S$ has the superamalgamation property 
in the next theorem
cannot be weakened, in general, to the strong amalgamation property.
See Example \ref{lopex} below.

\begin{theorem} \labbel{superap}
Suppose that $\mathcal S$ is a class of ordered structures
such that 
  \begin{enumerate} 
   \item  
$\mathcal S$ has the superamalgamation property, and
\item
every structure $\mathbf F  \in \mathcal S$ 
can be extended to some 
structure $\mathbf E \in \mathcal S$ such that 
every subset of $F$ has both a meet and a join in 
$\mathbf E$ 
(in particular, this applies if
every structure $\mathbf F$ in $\mathcal S$ 
can be embedded into some 
structure $\mathbf E \in \mathcal S$ such that
 the order on $\mathbf E$
is a 
complete bounded lattice).
 \end{enumerate}

If $\mathcal S_1$ is the class of expansions of structures 
of $\mathcal S$ obtained by adding
a new operation satisfying some fixed  property chosen from  (A1e) - 
 (C2)  
in Definition \ref{list},
then  $\mathcal S_1$ has the superamalgamation property,
in particular, the strong amalgamation property.
In particular, this applies to adding an isotone, or an extensive, idempotent,
 closure, interior, antitone\dots\   operation.

More generally, the same applies to expansions 
obtained by adding   families of such operations,
 possibly with comparability conditions
among operations satisfying the same property.
 \end{theorem} 

\begin{proof}
Suppose that $\mathbf A$, $\mathbf B$ and $\mathbf C$ 
are structures in $\mathcal S_1$ to be amalgamated
and with, say, a closure operation $K$
not in the type of $\mathcal S$.
By the superamalgamation property of $\mathcal S$,
the reducts $\mathbf A^-$, $\mathbf B^-$, $\mathbf C^-$
to the type of $\mathcal S$ can be superamalgamated to some
structure $\mathbf F^-$. 
By the assumption (2)  we can extend $\mathbf F^-$ to some  
$\mathbf E^-$ in $\mathcal S$ such that 
every subset of $F$, in particular, every subset of 
$A \cup B$ has both a meet and a join in $\mathbf E^-$.

We want to expand $\mathbf E^-$
to some structure $\mathbf E$ in  $\mathcal S_1$
in such a way that $\mathbf E$ strongly amalgamates 
$\mathbf A$ and  $\mathbf B$ over $\mathbf C$. 
If this is possible, $K_{\mathbf E}$ should agree
with the following function $G$ 
\begin{equation}\labbel{kd}
Gd =\begin{cases} 
K_{\mathbf A} d &    \text{if  $ d \in A  $},
\\ 
K_{\mathbf B} d &    \text{if  $ d \in B  $}
 \end{cases} 
\end{equation}  
defined on $D=A \cup B$. 
Notice that $K_{\mathbf A} $ and 
$K_{\mathbf B} $ agree on $C=A \cap B$, 
 by the assumptions in the hypothesis of the 
strong amalgamation property.

We shall use the superamalgamation property to check that 
$G$, as given by \eqref{kd}, satisfies the assumptions
of Lemma \ref{exte}, in this specific
instance,  extensiveness and the condition \eqref{isoide}. 
$G$ is obviously extensive, since both
$K_{\mathbf A} $ and 
$K_{\mathbf B} $ are.
 
To prove \eqref{isoide}, first assume that  
 $a \in A \setminus B$
and $b \in B \setminus A$.  
If $a \leq Gb$ in $\mathbf E^-$, then 
$a \leq Gb$ in $\mathbf F^-$, as well, 
since $Gb = K_{\mathbf B} b \in B \subseteq F^-$.
Since $\mathbf F^-$ superamalgamates 
$\mathbf A^-$ and $\mathbf B^-$ over $\mathbf C^-$,
there is  $c \in C$ such that 
$a \leq _{ \mathbf A}  c  $ and
 $c \leq _{ \mathbf B} Gb$.  
Then  $Ga= K_{\mathbf A}a
\leq_{\mathbf A} K_{\mathbf A}c = Gc$, since $a , c \in A$
and  $K_{\mathbf A}$ is a closure operation on $\mathbf A$.
Similarly, from $ c \leq_{\mathbf B} Gb$,
 that is, $ c \leq_{\mathbf B} K_{\mathbf B}b$,
we get  $Gc = K_{\mathbf B} c \leq_{\mathbf B} 
K_{\mathbf B}K_{\mathbf B}b = K_{\mathbf B}b =Gb$,    
 since
$K_{\mathbf B}$ is a closure operation on $\mathbf B$.
Since we assume that the embeddings 
from $\mathbf A^-$ and $\mathbf B^-$ to $\mathbf E^-$
are at least order-embeddings, 
from $Ga \leq_{\mathbf A} Gc$ and
$Gc \leq_{\mathbf B} Gb$
we get 
$Ga \leq Gc \leq Gb$ in $\mathbf E^-$,
hence $Ga \leq Gb$ by transitivity of $\leq$.

 The case  
$b\in A \setminus B$, $a \in B \setminus A$
is symmetrical, while the cases when 
$a,b\in A$ or $a,b \in B$ follow from the assumption
that $K_{\mathbf A}$, $K_{\mathbf B}$ are closure operations
 on $\mathbf A$, $\mathbf B$.
 We have proved \eqref{isoide} for $G$.    

We have showed that $D$ and $G$,
as chosen, satisfy the assumptions
in Lemma \ref{exte}(B3), hence  $G$ can be extended to
a closure operation $K$  on the whole of $E$.
By \eqref{kd},  the expansion of $\mathbf E^-$
obtained by adding $K$  superamalgamates    
$\mathbf A$ and $\mathbf B$ over $\mathbf C$, 
since we already know that $\mathbf E^-$  superamalgamates   
$\mathbf A^-$ and  $\mathbf B^-$ over $\mathbf C^-$.

The other cases in (A1e) -  (C2)    
 are entirely similar.
In cases (C1) - (C2) take $X= A^n \cup B^n$.
The point is that, in  the conditions
\eqref{iso} - \eqref{ison},  
each inequality involves only one element on the right and one
element on the left, so that we can apply the superamalgamation property.
 In cases (A1e) - (A3) it is enough to assume
the strong amalgamation property (this is
necessary for $G$ to be well-defined). 
The proof is similar to the above arguments.
For case (A3) notice that if both 
$K_{\mathbf A}$ and $K_{\mathbf B}$ 
are involutions,
$a \in A $,
$b \in B $ and $Ga=Gb$,
then $Ga=Gb \in A \cap B = C$, say,
$Ga=Gb=c \in C$.
Then $a=K_{\mathbf A}K_{\mathbf A}a=
K_{\mathbf A}Ga=K_{\mathbf A}c=K_{\mathbf B}c=
K_{\mathbf B}K_{\mathbf B}b=b$.  
Thus \eqref{invb} holds in $D$, for $G$
defined by \eqref{kd}.   

In passing, we remark that in a further work
we shall see  that, when dealing with
 additive operations,  
we  get conditions involving more than one element
on the right-hand side of the inequalities.
We can adapt the above arguments anyway, by assuming a
notion  stronger than superamalgamation. We shall present details elsewhere.

Finally, the construction of  $\mathbf E^-$
does not depend on $K$, or on the other additional operations,
hence we can repeat the above argument for as many operations
as we want. Comparability conditions
between operations satisfying the same property
are  preserved by
Definition \eqref{kd}, 
hence can
be  
maintained in view of Lemma \ref{compa} 
  and Remark \ref{molt}.   
Notice that the additional condition in the second sentence of
 Lemma \ref{compa}(1) is verified here, since
$D=A \cup B$, hence $d \in D$ implies 
$Gd \in D$, where $G$ is defined as in \eqref{kd}.  
\end{proof}  

As in the last paragraph  of Lemma \ref{exte},
we only need assume that 
every subset of $F$ has  a meet in 
$\mathbf E$ in clause (2) in Theorem \ref{superap}, unless we deal with
cases (B1c) or (B4).
 Theorem \ref{superap} holds also in case (C3),
under the assumption that 
$\mathcal S$ is a class of lattice-ordered structures
and that embeddings preserve the lattice operations.

In the following examples we show that the assumption (1)
is necessary in Theorem \ref{superap}.

\begin{examples} \labbel{lopex}
 Recall that a class of structures closed under isomorphism
has the \emph{amalgamation property} (AP) if, under the assumptions
in Definition \ref{sap}, we only obtain the weaker conclusion
that there are a model $\mathbf  D$ and  embeddings 
$\iota: \mathbf A \to \mathbf  D$ and
$\kappa: \mathbf B \to \mathbf  D$ 
which agree on $C$.
 The difference is that we do not necessarily assume 
that $\iota$ and $\kappa$ are inclusions, in other words,
 possibly, some 
elements of $ A$  need to be identified
with elements of $B$.

(a) It is then elementary to see that we need  the strong
amalgamation property in the hypothesis (1) of Theorem \ref{superap};
AP alone does not suffice. Indeed, if SAP fails, 
additional operations might behave differently on elements to be
identified,
hence it is not possible to embed both $\mathbf A$ and $\mathbf  B$ 
 into the same structure. For example, let $\mathbf  C$ be a distributive
lattice with some
operation $K$ 
and with some element $c \in C$ which has no complement in $\mathbf  C$.
Suppose that $\mathbf A$ and $\mathbf  B$ are extensions 
of $\mathbf  C$ in which $c$ has a complement, call 
such complements  $a$
and $b$,
respectively. In any amalgamating structure in the class
of distributive lattices, $a$ and  $b $ should be identified,
since complements are unique in distributive lattices.
But if, say, $Ka=a$ in $\mathbf A$ and
$Kb \neq b$ in $\mathbf  B$, 
then it is not possible to embed $\mathbf A$ and $\mathbf  B$
into the same structure.
The argument applies to most kinds of operations;
exceptional cases occur only when there are very
tight assumptions on $K$. 
For instance, if in the above example we assume that $K$ 
is an isotone involution, then necessarily
both $Ka=a$ and $Kb=b$, hence amalgamation
is possible.  

The above argument also explains why we need
to deal with embeddings.
The argument shows that we need to deal with, at least, 
injective homomorphisms. However,
the class of posets with injective order preserving
functions does not have AP, in the categorical sense
from \cite{KMPT}. Indeed, if $\mathbf  C$ 
has two incomparable elements $c$ and $d$,
we set $c \leq d$ in $\mathbf A$ (this is compatible
with the assumption that we deal with injective order preserving functions)
and $ d \leq c$    in $\mathbf  B$, then 
(the images of) $c$ and $d$ should be equal in   
any amalgamating structure $\mathbf  D$, thus injectivity
is lost.

(b) A more involved example
shows that we do need the superamalgamation property  
in the hypothesis (1) of Theorem \ref{superap};
SAP is not enough.  
The classes of linearly ordered sets 
(linearly ordered sets with one isotone operation) have 
the strong amalgamation property 
 \cite[Theorem 3.1(a)]{lop},   
 but not the 
superamalgamation property.
Every linearly ordered set can be embedded into a 
complete bounded linearly ordered set, hence 
every linearly ordered set with one isotone operation
can be embedded into a complete bounded linearly ordered set
with one isotone operation, by Corollary \ref{plap}(1).

On the other hand,  the class of linearly ordered sets with
two isotone operations does not have the amalgamation property 
 \cite[Theorem 3.1(c)]{lop}. 
This example shows that assumption (1) is necessary in
Theorem \ref{superap}
 and cannot be weakened to the strong amalgamation property:    
take $\mathcal S$ 
to be the class of  linearly ordered sets with one isotone operation
 and let $\mathcal S_1$
be obtained by adding another isotone operation.
  
The proof of Theorem 3.1(c) in \cite{lop}
actually gives counterexamples for all cases
(B1) - (B4). Indeed, the example in (c)(i) there provides
 a triple of  linearly ordered sets with two closure operations and 
 which has no amalgamating model in the class of
linearly ordered sets with two  isotone operations.
By \cite[Lemma 5.1]{lop}, the classes of linearly ordered sets
with an isotone and extensive (isotone and idempotent,
closure)  operation have the strong amalgamation property.
Notice that in \cite{lop} we used different terminology:
isotone operations are called \emph{order preserving} there,
and we used \emph{increasing} in place of extensive.   
Arguing as above, the counterexample in \cite[Theorem 3.1(c)]{lop}
takes care simultaneously
of (B1), (B1e), (B2), (B3), while (B1c), (B4) are dual.

(c) To deal with case (B5), in \cite[Remark 4.4]{lop}
we noticed that the class of linearly ordered sets with
an antitone operation with a fixed point (called a \emph{center} in \cite{lop})
has the strong amalgamation property. On the other hand,
  the class of linearly ordered sets with
two antitone operations, even with a common fixed point,
 does not have  AP (\cite[Theorem 4.3(b)]{lop}). Then argue as above.

(d)  Notice that cases (B1), resp., (B1e), are the special unary
cases of (C1) with $i=1$, resp., of (C2). 
Moreover, for lattice ordered structures,
 case (C2) is the special case
$t(x_1, \dots, x_i) = x_1$  of (C3).
Since the counterexample in (b) above
is a linearly ordered set, in particular, a lattice, we get that 
the superamalgamation property  is necessary also
in cases (C1) - (C3).

(e) On the other hand, as we mentioned in the proof of Theorem \ref{superap},
the strong amalgamation property is sufficient for cases (A1e) - (A3).  
 \end{examples}

If  $\mathcal H$  is a class of finitely generated structures,
 a Fra\"\i ss\'e  limit of  $\mathcal H$ is 
a countable universal  homogeneous structure of age  $\mathcal H$.
For example, the ordered set of the rationals is the Fra\"\i ss\'e
limit of the class of finite linearly ordered sets,
the random graph is the Fra\"\i ss\'e limit
of the class of finite graphs.  
See \cite[Section 7.1]{H}  for  details.

The joint embedding property is a necessary condition for the
existence of a Fra\"\i ss\'e limit. Recall that a class $\mathcal H$ has the \emph{joint embedding  property}
(JEP) if, for every  
$\mathbf A, \mathbf B  \in \mathcal H$,
 there are a structure
$\mathbf E \in \mathcal H$ and  embeddings
$ \iota: \mathbf A \to \mathbf E$
and 
 $ \kappa: \mathbf B \to \mathbf E$.
 For classes of structures
in which it makes sense to consider empty structures,
for example, posets or semilattices with no constant in the
language, the amalgamation property implies JEP.
On the other hand, for example,
nontrivial    
Boolean algebras with additional operations generally do not have the
joint embedding property since  it may happen that $K0=0$
in some algebra of the class, while $K0 \neq 0$ in some other algebra.  

 However, given a class  $\mathcal S$ 
with the amalgamation property, setting $\mathbf A \sim \mathbf  B$ 
if $\mathbf A, \mathbf  B \in \mathcal S$ and $\mathbf A$,
$\mathbf  B$ can be embedded into a same member of $\mathcal S$,
we get an equivalence relation such that each 
 equivalence class has AP and JEP. If furthermore
$\mathcal S$ is closed under substructures in a language with at least one 
constant, then $\mathbf A \sim \mathbf  B$ 
if and only if   the $ \emptyset $-generated substructures 
of $\mathbf A$ and $\mathbf  B$ are isomorphic.  
For example, the class of 
 nontrivial 
Boolean algebras with a closure
operation satisfying $K0=0$ has the joint embedding property. 
Notice that if $K$ is a closure operation, then necessarily $K1=1$. 

In the next corollary embeddings are meant to preserve all the operations
of the structures.

\begin{corollary} \labbel{corsuperap}
Let (W) be any one of the properties
(A1e) -  (C2)  
 from  Definition \ref{list} and 
let $\mathcal  S$ be any one of the following classes:
the class of partially ordered sets, of meet semilattices, of join semilattices, 
of lattices, of Boolean algebras, of Heyting algebras. 
 In the last three cases we may also allow  (W) to be (C3).  
Then
the following hold.
  \begin{enumerate}
    \item  
If $\mathcal  S ^{\text{(W)}} $
is the class of  structures obtained from 
members of $\mathcal  S$ by adding 
a new operation satisfying
(W), then $\mathcal  S ^{\text{(W)}} $
has the superamalgamation property, in particular,
the strong amalgamation property. 
\item
More generally, the superamalgamation property is maintained
if we add families of  operations
satisfying possibly distinct properties  from (A1e) - (C3).
Further, we may possibly
add comparability conditions among operations satisfying
the same property.
\item
For each of the above classes, the 
class of their finite members has a Fra\"\i ss\'e 
limit, under the provisions
that  in (2) above
only a finite number of operations are added and
that,
 in the case of  Boolean 
and Heyting algebras, we consider a subclass
of $\mathcal  S ^{\text{(W)}} $ consisting
of structures having a
fixed (modulo isomorphism) $ \emptyset $-generated substructure.  
  \end{enumerate}
 \end{corollary}

 \begin{proof}
 Classical and well-known results 
\cite[Ch. 1, Theorems 10.6 and 10.7]{Ha}, 
\cite[Theorem 2.1]{Ko}, \cite[Theorem 2.39]{GM},
\cite[Theorem 2.3]{HD} assert 
that, for every choice of $\mathcal  S$ 
as in the statement, each member of $\mathcal  S$ 
can be embedded into a complete bounded lattice
 in the corresponding class.    
Hence the corollary follows from
Theorems
 \ref{josup} and  
 \ref{superap}.

Item (3) follows from Fra\"\i ss\'e Theorem
\cite[Theorem 7.1.2]{H}.
The finiteness assumption is needed in order to have only a countable
 number of nonisomorphic finite structures.   
Under the assumptions, the joint embedding property follows
from the amalgamation property, since it turns out to be equivalent
to the case when $\mathbf  C$ is the 
 fixed $ \emptyset $-generated substructure, 
 for
Boolean and Heyting algebras
 with operations,  
and since we can consider amalgamation over
an empty structure in the other cases.
\end{proof}  

Corollary \ref{corsuperap} applies with the same proof to the classes of 
bounded partially ordered sets, bounded  
lattices, bounded  meet semilattices, bounded  join semilattices
(if maxima and minima are required to be preserved
by embeddings, for example, when they are interpreted as constants, then in item (3)
we need 
 consider a subclass of  $\mathcal  S ^{\text{(W)}} $ with
a fixed $ \emptyset $-generated substructure).

As a way of example, the following classes 
have the superamalgamation property, and the classes of their
finite members have  a Fra\"\i ss\'e limit.
  \begin{enumerate}   
\item
The class of non-trivial Boolean algebras with three closure operations
$K$, $K_1$ and $K_2$ 
such that $K0=K_1 0=K_2 0=0$ and  $Kx \leq K_1x$,  
$Kx \leq K_2 x$ hold for every $x$. 
\item
The class of lattices with a closure operation, an antitone unary operation 
and a $3$-ary operation which is isotone on each component.
\item 
The class of   posets with two $4$-ary operations 
$F_1$ and  $F_2$ which are isotone on the first two components,
antitone   on the last two components and are such that 
$F_1(x,y,z,w) \leq F_2(x,y,z,w)$ always holds.
  \end{enumerate}

In the next section we shall prove that
the sets of  universal  
 consequences of the corresponding 
first-order theories are decidable in cases (1) and (3). This holds in case (2),
as well, if in place of lattices we consider distributive lattices. 

The list of those classes $\mathcal  S$ to which Theorem  \ref{superap}
applies is illustrative and not intended to be exhaustive. 

Recall   that some 
first-order theory $T$  has \emph{model companion}
if $T$ has the same universal consequences of some model complete
theory $T^*$. If in addition $T$ has the amalgamation property,
the theory $T^*$
is a \emph{model completion} of $T$. See \cite{H} for details.

\begin{definition} \labbel{lfdef}   
A universal theory $T$  is \emph{locally finite}
 if every finitely generated model of $T$ 
is  finite. If $T$ is universal in a finite language and 
$T$ is locally finite 
then $T$ is actually \emph{uniformly locally finite}, to the effect that 
there is a function $g_{_T} : \mathbb N \to \mathbb N$  such that, for every
$m \in \mathbb N$, every model of $T$ generated by $m$
elements has cardinality $\leq g_{_T}(m)$ \cite[Lemma 5]{We}.
In particular, such a $T$ has a finite number
of models generated by $m$ elements, since the language
of $T$ is finite.    
 \end{definition}

 The next proposition is folklore; compare \cite[Fact 2.1]{KS}.

\begin{proposition} \labbel{folk}
Suppose that $T$ is a consistent first-order locally finite universal theory
in a finite language. If $T$ has AP and JEP, 
then the class of finite models of $T$ 
 has a Fra\"\i ss\'e limit  $\mathbf M$.
The first-order theory $Th(\mathbf M)$ of $\mathbf M$
  is $\omega$-categorical,
 has quantifier elimination and is
the model completion of $T$.
 \end{proposition}

  \begin{proof} 
The proposition follows from \cite[Theorem 7.1.2 and Theorem 7.4.1]{H}.
The argument showing that $Th(\mathbf M)$
is
the model completion of $T$ can be found in \cite[Fact 2.1(3)]{KS}.
There the result is stated under the stronger assumption
of uniform local finiteness, but, under the hypotheses
of the proposition, it is equivalent
to local finiteness by the mentioned
Lemma 5 in \cite{We}.
\arxiv{

In more detail, since $T$ is universal and locally finite,
AP and JEP for $T$ imply
AP and JEP for the class of finite models of $T$.
The hereditary property holds since
$T$ is assumed to be universal.
Since $T$ locally finite  in a finite language, 
then $T$ has countably many finite models
up to isomorphism.
Thus
Fra\"\i ss\'e Theorem
\cite[Theorem 7.1.2]{H}
provides the existence of a Fra\"\i ss\'e limit $\mathbf M$.

As we mentioned in Definition \ref{lfdef},
$T$ is uniformly locally finite, by \cite[Lemma 5]{We}. 
 Then, by
\cite[Theorem 7.4.1]{H}, 
 $Th(\mathbf M)$ is $\omega$-categorical and
 has quantifier elimination, in particular,
$Th(\mathbf M)$ is model complete.
The model $\mathbf M$ is constructed as the union of a chain
of models of $T$, hence  $\mathbf M$ is a model of $T$,
since $T$ is universal.  
Hence the theory $Th(\mathbf M)$ contains  $T$.
Conversely, if some universal sentence $\varphi$   fails in some model of 
$T$, then $\varphi$   fails in a finite model of $T$, since $T$
is locally finite. But every finite model of $T$ can be embedded
in $\mathbf M$, since $\mathbf M$ is a universal model, thus $\varphi$  
fails in $\mathbf M$. Hence $T$ and $Th(\mathbf M)$
have the same universal consequences, and this means that 
$Th(\mathbf M)$ is the model completion of $T$, since 
$Th(\mathbf M)$ is model complete and $T$ has AP.} 
\end{proof}

\begin{corollary} \labbel{smk} 
The first-order theory $T$ of
join semilattices with a closure operation
has model completion.

In more detail, if $\mathbf M$
is the Fra\"\i ss\'e limit of the class of finite
join semilattices with a closure operation,
 then the first-order theory $Th(\mathbf M)$  is $\omega$-categorical,
 has quantifier elimination and is
the model completion of $T$.

Dually, the above results apply to meet semilattices 
with an interior operation.
\end{corollary}

 \begin{proof} 
 By Corollary \ref{corsuperap} 
the theory of join semilattices with a closure operation
has AP and, as mentioned in the proof,
this implies JEP, since we can consider
$\mathbf  C$ as an empty model.  
In the next lemma we show that the theory
of join semilattices with a closure operation
is locally finite.
 The result then follows from Proposition \ref{folk}. 
\end{proof}

\begin{lemma} \labbel{jsllf}
The theory
of join semilattices with a closure operation
is locally finite.
 \end{lemma} 

\begin{proof} 
Suppose that $\mathbf S$ is a  
join semilattice with a closure operation $K$ and   
suppose that $\mathbf S$ is generated by 
the elements $x_{1}, \dots, x_{n}$. 
We claim that each element of $S$ can be written in the form
\begin{equation}\labbel{semg}
\begin{aligned}
x_{j_1} \vee x_{j_2} \vee  \dots \vee x_{j_h}
&\vee K( x_{\ell_{1,1}} \vee x_{\ell_{1,2}} \vee  \dots \vee x_{\ell_{1,k(1)}})
\vee \dots
\\ 
& \vee
K( x_{\ell_{m,1}} \vee x_{\ell_{m,2}} \vee  \dots \vee x_{\ell_{m,k(m)}}),
 \end{aligned}    
\end{equation}    
with $j_1, \dots, j_h, \ell_{1,1}, \dots, \ell_{m,k(m)} \leq n$,
and
where possibly $h=0$,
that is, we have a join of  expressions with a closure, and
possibly $m=0$, that is, we have an expression without closures.
The expression \eqref{semg}
is not ambiguous because of associativity of $\vee$.
Because of commutativity and idempotence, we can assume that   
 the  $j_i$s are all distinct, and that, for each  
$p \leq m$, the indices $\ell_{p,1}, \dots, \ell_{p,k(p)}$
are all distinct. Moreover, we can assume that, letting $p$ vary,  the sets
$ \{\ell_{p,1}, \dots, \ell_{p,k(p)}  \} $ are all distinct.
Hence, up to semilattice equivalence, we have at most
$2^n + 2^{2^n}$ expressions of the form \eqref{semg}
(of course, this is an overestimated rough  bound).   

The join of two expressions of the form \eqref{semg}
has still the form \eqref{semg} and can be reduced
as above using associativity,  commutativity and idempotence
of $\vee$. It remains to show that  if 
$\sigma$ is an expression of the form \eqref{semg},
then $K \sigma $ can be reduced to the form \eqref{semg}
by using the properties of a closure in a join semilattice.
In fact we will show that 
\begin{equation}\labbel{top}   
K(a_1 \vee \dots \vee a_r \vee Kb_1 \vee \dots \vee Kb_s)
= K(a_1  \vee \dots \vee a_r \vee b_1 \vee \dots \vee b_s)
  \end{equation}   
holds in every  join semilattice with a closure operation,
 for all $a_1, \dots b_s$,  thus if we 
apply $K$ to \eqref{semg}, we get
$K(x_{j_1} \vee \dots \vee x_{\ell_{m,k(m)}})$,
a very special expression still of the form \eqref{semg}.

So let us prove \eqref{top}.   
Since $Kb_1 \geq b_1, \dots$, then
 $a_1 \vee \dots  \vee Kb_1 \vee \dots \vee Kb_s
\geq a_1 \vee \dots  \vee b_1 \vee \dots \vee b_s$,
thus, applying $K$ and by isotony, we get 
 $K(a_1 \vee \dots \vee Kb_s)
\geq K(a_1 \vee \dots \vee b_s)$.
For the converse, by extensiveness and isotony,
we have 
 $a_1 \leq K(a_1 \vee \dots \vee b_s)$, \dots,
 $Kb_s \leq K(a_1 \vee \dots \vee b_s)$,
hence 
 $a_1 \vee \dots \vee Kb_s
\leq K(a_1 \vee \dots \vee b_s)$.
Applying $K$, we get
 $K(a_1 \vee \dots \vee Kb_s)
\leq KK(a_1 \vee \dots \vee b_s)=
K(a_1 \vee \dots \vee b_s)$ by isotony and 
idempotence.
\end{proof}

Notice that, in contrast with
Lemma \ref{jsllf}, the theory of meet semilattices
with a closure operation is not locally finite.
 See \cite[Section 2]{Ja}, in particular, p. 3 and Figure 1 
on p. 13 therein.

\section{Decidability of  universal  
 consequences} \labbel{lfsec} 

If $T$ is a universal locally finite theory,
then a 
 universal-existential sentence $\varphi$   is a consequence of $T$ 
if and only if $\varphi$  holds in every finite model of $T$.
If we extend $T$ in a language with added operations,
then the extended theory $T^+$ is not necessarily locally finite.
However, we can retain the above characterization,
 limited to  universal consequences,
when we add  operations of the kind considered in 
the present note and every finite model of $T$ 
 can be extended to a finite lattice-ordered model.  
Compare \cite[Appendix IV]{MT} 
for a special similar situation.

The present section relies only on Section \ref{extsec}
and does not deal with the amalgamation property. 

\begin{theorem} \labbel{decid}
Suppose that $T$ is a locally finite universal theory 
in a language $\mathscr L$ with a specified order relation 
$\leq$ and suppose that 
 every finite model of $T$ 
 can be extended to a finite lattice-ordered model of $T$. 

Suppose that (W) is any one of the properties
(A1e) -  (C2)  
 listed in  Definition \ref{list} and
$\mathscr L'=\mathscr L \cup \{ K \} $,
where $K$ is a new  operation symbol of corresponding arity.
 Let $T^{\text{(W)}}$ in the language $\mathscr L'$
be the extension of $T$  obtained by
adding axioms   asserting that $K$ satisfies (W). Then
the following hold. 
  \begin{enumerate}   
 \item 
If $\varphi$  is a  universal  
 sentence in $\mathscr L'$
and $\varphi$  fails in some model of  $T^{\text{(W)}}$,
then $\varphi$  fails in some finite model of  $T^{\text{(W)}}$.
\item
 Suppose that  $\mathscr L'$ is finite and there is an 
effectively computable  function 
$h_{_T}$ such that, for every $n \in \mathbb N$,
every model of $T$ generated by  $ n$ elements  
can be extended to a  lattice-ordered model of $T$
of cardinality $\leq h_{_T}(n)$.
Then the set of all the 
universal 
 consequences
of  $T^{\text{(W)}}$ is decidable.  
  \end{enumerate}

More generally, the above items (1) - (2) hold if we add 
any number of operations, possibly of distinct arities, 
possibly satisfying distinct
properties chosen from (A1e) -   (C2).
If $T$ contains the axioms for (and in the language of) lattices, then 
(W) might be chosen to be (C3), too.  
 \end{theorem} 

\begin{proof}
(1) Suppose that $\varphi $ is $  \forall \bar x \psi$,
with $\psi$  quantifier-free,   
and $\varphi$  fails in some model of 
$T^{\text{(W)}}$. 
If some term of the form $K(t_{1}, \dots, t_{n})$
occurs in $\varphi$ and $K$ 
does not occur in the terms $t_{1}, \dots, t_{n}$, let 
 $y$ be a new variable not occurring in $\varphi$.
Then $\varphi$  is logically equivalent to 
$\forall \bar x y( K(t_{1}, \dots, t_{n}) {\,=\,} y
\Rightarrow  \psi^* )$, where
$\psi^*$ is obtained from $\psi$
by substituting all the occurrences of the term $K(t_{1}, \dots, t_{n})$ 
for $y$.    
Iterating the above procedure, it is no loss of generality
to assume that 
 $\varphi$   is of the form  $\forall \bar x \bar y \psi$, where   
\begin{equation}\labbel{form}   
\psi : \quad  K(t_{1,1}, \dots, t_{1,n}){\,=\,} y_1\,\& \dots \& \,
 K(t_{m,1}, \dots, t_{m,n}) {\,=\,} y_m \implies  \psi^* 
\end{equation}    
with $t_{1,1}, \dots, t_{m,n}, \psi^*$ $K$-free and 
$\psi^*$  quantifier-free.

 By assumption, there is some model $\mathbf A$ of 
$T^{\text{(W)}}$ such that $\varphi$  fails, hence,
for an appropriate assignment of elements of $A$ 
to the variables of $\psi$, 
the evaluation of $\psi$ fails in $\mathbf A$. 
This means that, for the given assignment, 
$K(t_{1,1}, \dots, t_{1,n}){\,=\,} y_1, \dots, \allowbreak  
 K(t_{m,1}, \dots, t_{m,n}) {\,=\,} y_m$ hold and
 $\psi^*$ fails in $\mathbf A$.
Let $t_{i,j} ^{\mathbf A} $ 
denote the evaluation of $t_{i,j}$ under the assignment and
let $X  $ be the set of the $n$-tuples of $A$ 
having the form   $(t_{j,1}^{\mathbf A}, \dots, t_{j,n}^{\mathbf A}) $,
 for $1 \leq j \leq m$.  
Let $V:X \to A $ be defined by 
$V(t_{j,1}^{\mathbf A}, \dots, t_{j,n}^{\mathbf A})= 
K _{\mathbf A}(t_{j,1}^{\mathbf A}, \dots, t_{j,n}^{\mathbf A})$,
which is also equal to $y_j ^{\mathbf A} $,
since $K (t_{j,1}, \dots, t_{j,n})=y_j$ holds in $\mathbf A$. 
Notice that in the unary case $X$ and $V$ are called $D$ and $G$ in
Lemma \ref{exte}.

 Since $\mathbf A$ is a model of 
$T^{\text{(W)}}$ and $K _{\mathbf A}$ 
is an extension of $V$ satisfying (W), then the necessary condition
in Lemma \ref{exte} for
the satisfaction of property (W) holds
(as we mentioned at the beginning of the proof of Lemma \ref{exte},
no completeness assumption is needed to prove the necessary condition).

 Let $\mathbf B^-$ be the subreduct  
 of $\mathbf A$  generated 
in the language $\mathscr L$ by
 the elements assigned to the variables of $\psi$ under the given assignment.
Thus $\mathbf B^-$ is a finite model of $T$, since $T$ is universal
and  locally finite.
By construction, $X \subseteq (B^-)^n$
and  $V$ is actually a function from $X$ to $B^-$. 
By assumption, $\mathbf B^-$ 
 can be extended to a finite lattice-ordered model
$\mathbf  C^-$  of $T$.
Hence $\mathbf  C^-$ is complete, since $C^-$ is finite.    
We can apply Lemma \ref{exte}
in order to extend $V$ on the whole of 
 $(C^-)^n$   
to an operation
$K_{\mathbf C}$ in such a way that 
$K_{\mathbf C}$ 
extends $V$ and (W) holds in the expanded model 
  $\mathbf C$.  
Thus 
  $\mathbf C$  
 is a 
model of $T^{\text{(W)}}$,
since 
  $\mathbf C^-$  
  is a model of $T$. 

 We have that $t_{i,j} ^{\mathbf A} = 
t_{i,j} ^{\mathbf B^-}=t_{i,j} ^{\mathbf C}$ hold,  
 for all
pairs of indices, since the terms $t_{i,j}$ are $K$-free,
 since
  the variables of $\psi$ are interpreted in $B^-$,
because of the definition of $\mathbf  B^-$,
and since $ C \supseteq B$.  
Since $K_{\mathbf C}$ 
extends $V$ and 
because of the definition of $V$,
$K(t_{1,1}, \dots, t_{1,n}){\,=\,} y_1, \dots, \allowbreak 
 K(t_{m,1}, \dots, \allowbreak t_{m,n}) {\,=\,} y_m$ hold 
in
  $\mathbf C$  
 under the given assignment.

On the other hand, since $\psi^*$  is $K$-free,
 quantifier-free  
 and false in $\mathbf A$, then
 by the definitions of $\mathbf B^-$ and   $\mathbf C$,   
$\psi^*$ is false in   $\mathbf C$.  
This shows that $\varphi$  is false in   $\mathbf C$,  
 thus 
$\varphi$  fails in a finite model of $T^{\text{(W)}}$.

(2) Let $\varphi$  be a  universal  
sentence in the language of 
$T^{\text{(W)}}$.
The proof of (1) shows that $\varphi$ fails in some model of
$T^{\text{(W)}}$ if and only if 
$\varphi$  fails in some  lattice-ordered finite model of
$T^{\text{(W)}}$ whose $\mathscr L$-reduct extends
a model of $T$ generated by $k$ elements, where  
 $k$  can be effectively determined
and depends only on the formula $\varphi$.
In fact, if   $\varphi$ contains $ \ell$ variables   
 and $K$  occurs $m$ times
in $\varphi$, then $k \leq  \ell + m$. 

Thus $\varphi$  is a consequence of  $T^{\text{(W)}}$
if and only if $\varphi$  holds in every model of $T^{\text{(W)}}$
of cardinality 
 $\leq h_{_T}(k)$.  Since $h_{_T}$  
 is effectively computable
and the language of $T^{\text{(W)}}$ is finite,
one can effectively check the validity of $\varphi$  in all
these models. This  provides
a decision procedure for the validity of $\varphi$
in all models of $T^{\text{(W)}}$.

The last 
 paragraph  
 in the theorem is proved in the same way.
In the general case, the premises in $\psi$ in \eqref{form}
might involve distinct operations, but we can always manage
to have all the terms $t_{i,j}$ to be $\mathscr L$-terms.
In (2) the extended language  is finite by assumption; as far as (1)
is concerned, notice that a first-order formula involves only a 
finite set of symbols; then in 
 $\mathbf C$  
 all the remaining
symbols can be interpreted in an arbitrary way, for example, 
 as the projection onto the first component. Notice that
in the case of many operations $T^{\text{(W)}}$
says nothing about the mutual relationships among the operations.
To prove the last statement, observe that
if $T$ is a theory of lattices in the language of lattices,
then every lattice term is evaluated in the same way in
$\mathbf A$ and $\mathbf  B^- $, which in the present
situation can be
taken as $ \mathbf  C^-$.
Hence any condition of the form 
$t(x_1, \dots, x_i) \leq V(\bar x)$ is preserved.   
 \end{proof}

\begin{corollary} \labbel{badecid} 
 Let $T$ be the extension of the theory of Boolean algebras
in a language  with a further finite set
of operations, and with further axioms asserting 
that each operation satisfies some condition  
chosen among (A1e) - (C3) from Definition \ref{list}.
Then the  set of  universal  
 consequences of $T$
is decidable.

 The same applies to  distributive lattices 
in place
of Boolean algebras,  more generally,
to any locally finite universal theory\footnote{ provided the function
$g_{_T}$ from Definition \ref{lfdef} is effectively computable.} of
 lattices, and,
excluding case (C3),  
to partially ordered sets,
  join semilattices,  meet semilattices.
\end{corollary}

 In many cases Corollary \ref{badecid},
as well as  the last paragraph in the statement of Theorem \ref{decid},
can be generalized by adding comparability
conditions among operations satisfying the same property.
We leave details to the reader.  

Generally, for a  theory $T$ as in Corollary \ref{badecid}, the set of 
\emph{all} the first order consequences
of $T$ is not
 decidable. Indeed, the set of 
 the first order consequences of
the theory of Boolean algebras with an additive
closure operation (called \emph{closure algebras} in the literature)
is not decidable
\cite[footnote 19]{MT}, \cite{Gr}.
Were the consequences of a theory $T$ as in Corollary \ref{badecid}
decidable (except,  possibly,  
 for the cases of an involution and
of an antitone operation),  we could add as a premise a finite set of
sentences characterizing closure algebras, which would produce
a decision procedure for the consequences of the theory of closure
algebras, a contradiction. Notice that the property that, say, a poset
is (the order-reduct of) a Boolean algebra can be expressed by a first-order
sentence in the language of posets.

As another observation, notice that
 the proof of  
 Corollary \ref{badecid} does not apply 
to the theory of lattices, which is not locally finite.
On the other hand, the results in Section \ref{sapimp} 
 do not apply to distributive lattices, which have the 
amalgamation property but not the strong amalgamation property \cite{FG}.
This implies that the amalgamation property is 
generally destroyed by
adding further operations, 
 as exemplified in 
Example \ref{lopex}(a).

\section{Further remarks} \labbel{fur}

\begin{remark} \labbel{HHrem}
(a) The assumption that $HHx \leq Hx$, for every $x \in P$,
is necessary in   Lemma \ref{HH}. Consider the 3-element chain
$P= \{  a,b,c\} $ with $a < b <c$.
Let $Ha=b$, $Hb=Hc=c$, thus $H$ is isotone,
but $c=HHa \centernot \leq Ha=b$.
Let  $K_1 a=K_1 b=b$, $K_1 c=c$,
 $K_2 a=a $, $ K_2 b=K_2 c=c$,
thus $K_1$ and $K_2$ are both isotone, idempotent and smaller than $H$.
However, the only idempotent operation
larger than both   $K_1$ and $K_2$ is the constant function with value $c$,
which is not smaller than $H$. 

Notice that, in the above example, 
 the operations $K_1$ and $K_2$  are also extensive.
Thus in  Lemma \ref{HH} the assumption $HHx \leq Hx$
is necessary also in the extensive case
(in which case the assumption reads $HHx = Hx$,
hence in this case the Lemma is trivially proved by taking $K=H$).

(b) The assumption that every nonempty infinite
chain has a meet in $\mathbf P$ is necessary in 
Lemma \ref{HH}. For example, if $\mathbf P$ is the
ordered  set $\mathbb Z$
of the integers and $H$ is the predecessor function,
then in $\mathbf P$ there is no  idempotent operation
 smaller than $H$. 

(c) If in the above example we add a minimum $-\infty$
to $\mathbb Z$ and set $H(-\infty)=-\infty$,
  the assumptions in  Lemma \ref{HH} are met.
The example of $\mathbb Z \cup \{ -\infty \} $ shows 
that in the proof of Lemma \ref{HH} a finite iteration of 
the $K^ \alpha $s 
 is generally not sufficient.
 
(d) Some completeness assumption is necessary in 
Proposition \ref{minb2}. Again on $\mathbb Z$, define
\begin{equation*}      
 K_1 x=\begin{cases}
   x-1 & \text{if  $ x  $ is even},\\
 x   & \text{if  $ x  $ is odd,}
\end{cases}
\qquad
 K_2 x=\begin{cases}
   x & \text{if  $ x  $ is even},\\
 x-1   & \text{if  $ x  $ is odd.}
\end{cases}
  \end{equation*}
 Both $K_1$ and $K_2$
are isotone and idempotent, but  
on $\mathbb Z$   there is no  idempotent operation
 smaller than both $K_1$ and $K_2$.
 \end{remark}

\begin{remark} \labbel{larger}
(a) In  the cases   (A1c), 
 (B1), (B1e), (B2), (B3) and
(B5)  in Lemma \ref{exte}
there exists the largest operation $K$  satisfying the conclusions.
Recall that  we say that some operation $K$ is \emph{larger} than
$J$ if $Kx \geq Jx$, for every $x$ in the domain. 
The largest operation is given
by the corresponding formulae in  the proof of Lemma \ref{exte}.
 For posets with a maximum, the largest operation
exists in cases (A1e) and (A2e), as well. In case  (A1e)   
 set
$Ka=Ga$ if $a \in D$
and $Ka$ to be the maximum of $\mathbf P$, otherwise.
 In case (A2e) set
$Ka=Ga$ if $a \in D$,
$Ka=a$ if $a=Gb$, for some $b \in D$,   
and $Ka$ to be the maximum of $\mathbf P$ in the remaining cases.  

Dually, in cases   (A1e),  (B1), (B1c), (B2), (B4), (B5)
there is the smallest operation,  given by the dual formulae.
 
(b) On the other hand, 
in case (A2) 
there does not necessarily exist the largest operation
satisfying the conclusion in Lemma \ref{exte}.
Consider a five elements lattice with maximum $1$, minimum $0$
and three more elements $a,b,c$ such that 
$a \vee b= 1$ and $a \wedge b= c$
(a ``diamond'' with a new bottom element added).
If $D= \{ 0,1,a,b \} $ and $G1=G0=0$, $Ga=a$ and
$Gb=b$, then we can extend   $G$ to an idempotent operation
by taking $Kc \in \{ 0, c,a, b \} $, but we cannot set
$Kc=1$, if $K$ extends $G$ and is idempotent. 
Hence there is no largest idempotent operation extending $G$.

 (c) In general, the largest operation does not exist
in case (A3), either. Consider a ``diamond'' with maximum $1$,
minimum $0$ and $a$, $b$ such that  
$a \vee b= 1$ and $a \wedge b= 0$. Let 
$D= \{ 1 \} $ and $G1=1$.  
 If  $K_{{\circ}}1=1$, $K_{{\circ}}b=0$,
 $K_{{\circ}}0=b$ and
  $K_{{\circ}}a=a$, then
  $K_{{\circ}}$ is an involution extending $G$.
Similarly, setting
$  K_{{\bullet}}a =0 $,
$  K_{{\bullet}}0 =a $,
$  K_{{\bullet}}b =b $ and
$  K_{{\bullet}}1 =1 $, we get
an involution extending $G$.
If $K$ is larger than both 
$K_{{\circ}}$ and $K_{{\bullet}}$,
then $K0=1$, 
thus K is not an involution, if $K$ extends $G$.  
\end{remark}

\begin{remark} \labbel{basta}
(a) We do not need the full assumption that $\mathbf P$  
is a bounded and complete lattice in
cases (B1e) and (B3) in
 Lemma \ref{exte}.
It is enough to assume that $\mathbf P$ 
is a poset such that, for every $x \in P$,
every subset of $ \{ \, y \in R  \mid x \leq y \, \} $
has a meet, where $R$ is the range of $D$. This
is some kind of a near-lattice completion.

If  $R$  
 is cofinal in $\mathbf P$, that is,
for every $x \in P$, there is  $y \in R$  
such that $x \leq y$, then     
it is enough to assume that, for every $x \in P$,
every nonempty subset of $ \{ \, y \in R  \mid x \leq y \, \} $
 has a meet.

The dual assumptions are enough to deal with 
cases (B1c) and (B4).

The completeness assumption can thus be weakened
as above in the corresponding cases in Corollary \ref{plap}(1)
and Theorem \ref{superap}.

(b)
On the other hand, some completeness assumption is
necessary in Lemma \ref{exte}, even in case (B3).   

 Consider a  poset $\mathbf P$  with a descending chain
$( c_n) _{n \in \mathbb N} $  
and three elements $a,b,d$
smaller than all the $c_i$s
and such that $a < b$, $a < d$
and $b$,  $d$ incomparable.     

Let $D= P \setminus \{ b \} $ and let 
$G:D \to P$ be defined by $Ga=d$
and $Gx=x$, for $x \in D \setminus \{ a \} $.
The function $G$ satisfies    \eqref{isoide};
however, $G$ cannot be extended to a closure operation
$K$ on the whole of $P$.
Indeed, since  $a< b $,  
we should have $ Kb \geq Ka=Ga=d$.
But we also want $Kb \geq b$, hence, since
$b$ and $d$ are incomparable, then $Kb=c_i$,
for some i. Then $c_i= Kb \leq  Kc _{i+1} = c _{i+1} $,
since  $b \leq c_{i+1}$  
 and $K$ should be isotone.
This is a contradiction, since we have assumed
 $c _{i+1} <c _{i} $. 

The counterexample works also for case (B1e),
since we have not used idempotence.      

(c) In the above counterexample $\mathbf P$
is not a lattice, but a more involved counterexample can be 
devised to treat the case when $\mathbf P$ is a bounded 
 (necessarily incomplete)    
 lattice.  The following example
has also the advantage of working for all cases 
(B1) - (B4)  

Let $F$ be the set of all the  finite subsets of $\mathbb N$,
$p$ the set of even natural numbers and 
$P= F \cup \{ \, p \cup f \mid   f \in F\, \} 
\cup \{ \,  c_i \mid  i \in \mathbb N\, \} $,
where the order among the subsets of $\mathbb N$
is inclusion and the $c_i$s are a descending chain
of new elements taken to be greater than all  the subsets of $\mathbb N$.
Thus $P$ becomes a bounded distributive lattice with maximum $c_0$
and minimum the empty subset of  $\mathbb N$.  

Set $D=F \cup \{ \,  c_i \mid  i \in \mathbb N\, \} $ 
and $Gf=[0, \max f]$, 
for $f \in F$,  and $Gc_i=c_i$, for every $i \in \mathbb N$.
Then  $G$ is extensive and satisfies \eqref{iso} - \eqref{isoidebis}.
On the other hand,  
$G$ cannot be extended to an  isotone  
 operation $K$ 
on $P$, since $p \geq \{ \, 0, 2, \dots, 2n   \, \} $,
for every $n \in \mathbb N$, hence we should have 
$Kp \geq K\{ \, 0, 2, \dots, 2n   \, \}=[0,2n] $,
for every $n \in \mathbb N$, hence $Kp = c_i$, for some $i$,
but then we get a contradiction arguing as in (b).

 In the present counterexample we have only used isotony of $K$,
hence the counterexample (or its dual) applies to all cases
(B1) - (B4). Moreover, $D$ and $G$ satisfy the stronger
condition that if $x \in D$, then $Gx \in D$.

(d) The counterexample in (c) can be adapted in order to work for case
(B5).  Let $\mathbf P^+$ be the lattice described in (c);
let $ P^- = \{ \, x^- \mid x \in P^+ \, \} $ be a disjoint copy of 
$ P^+$  endowed with the reversed order and
set $ P^*= P^+ \cup  P^-$, letting every element
of $P^-$ be smaller than every element of $P^+$.
Let $D^+=D$ as introduced in (c) and $D^-$ correspond to the copy of $D$
in $P^-$; then set $D^*= D^+ \cup D^-$.
Given the function $G$ introduced in (c),
let $G^*:D^*\to P^*$ be the function defined by
$G^*x = (Gx)^-$, for $x \in P^+$  and
$G^*(x^-) = Gx$, for $x^- \in P^-$.
Since $G$ satisfies \eqref{iso}, then  $G^*$
satisfies \eqref{ant}. An argument similar to the one in 
(c) shows that     $G^*$ cannot be extended to an antitone
operation on $\mathbf P^*$.

(e) By the comment in Example \ref{lopex}(d), 
and since the counterexample in (c) above
is a lattice,  some completeness assumption
is necessary in cases (C1) - (C3).
  
 \end{remark}

\begin{remark} \labbel{compa2}
 (a) In Lemma \ref{compa} case (A3) the comparability condition is not
necessarily preserved,
unless  the additional assumptions in Lemma \ref{compa} are
satisfied. Consider the four element chain $\mathbf P$ with
$a<b<c <d $, 
$D= \{  a,b \} $,
$G_{{\circ}}a=b$,
$G_{{\circ}}b=a$,
$  G_{{\bullet}}a =c $ and $  G_{{\bullet}} b=d$.

We have $G_{{\circ}}x \leq G_{{\bullet}}x$,
for $x \in D$. Moreover both $G_{{\circ}}$
and  $G_{{\bullet}}$, taken alone, satisfy
the conditions   \eqref{inv} and \eqref{invb}, hence, by Lemma \ref{exte},
both $G_{{\circ}}$
and  $G_{{\bullet}}$ can be extended to some involution.
However, it is not possible to extend them in such a way
that the comparability condition $K_{{\circ}}x \leq K_{{\bullet}}x$
is satisfied, since involutions are bijective, hence we must have
$K_{{\circ}}c \geq c$; on the other hand, since  $K_{{\bullet}}$
must be an involution extending $G_{{\bullet}}$, then
$K_{{\bullet}}c=a$, thus necessarily   $K_{{\bullet}}c < K_{{\circ}}c$.

Notice that in the above example we have 
that $x \in D$ implies $G_{{\circ}}x \in D$
(not so for $G_{{\bullet}}$, of course, otherwise 
Lemma \ref{compa} would be contradicted).    

(b)  Without further assumptions, in cases (B2) and    
(B3)
 comparability conditions
are not necessarily preserved \emph{by the operations
defined  in the proof of Lemma \ref{exte}} 
 (but, as shown in Lemma \ref{compa},  we can maintain
comparability by introducing different operations).  

Let $P= \{ d,p,q \} $ with $d<p<q$,
$D= \{ d \} $ 
and let  $G_{{\circ}}d=d$,  $G_{{\bullet}}d=p$.
We have $G_{{\circ}}d \leq G_{{\bullet}}d$;
however, if $K_{{\circ}}$ and $K_{{\bullet}}$ are correspondingly defined 
by \eqref{gkclo}, then  $K_{{\circ}}p = q \not\leq p =K_{{\bullet}} p$.
 Similarly, if $H_{{\circ}}$ and $H_{{\bullet}}$ are
correspondingly defined by \eqref{kisoide1} in the
proof of case (B2), then  
$H_{{\circ}}d=d$, $H_{{\circ}} p=H_{{\circ}} q = q$
and
$H_{{\bullet}}d=H_{{\bullet}}p =p$, $H_{{\bullet}}q =q$.
In both cases $HHx=Hx$, for every $x$,
hence no iteration is needed, and   
$ K_{{\circ}}p= H_{{\circ}}p = q \not\leq p =H_{{\bullet}}p = K_{{\bullet}} p$,
thus the example works also for case (B2).

\arxiv{(c) In Lemma \ref{compa}, cases (B2) and (B3), 
the comparability conditions are satisfied 
by the operations
defined in the proof of Lemma \ref{exte},
under the additional assumption  that
$G_{{\bullet}}b \in D$, for every $b \in D$.
In the dual  case (B4) we need  assume instead that
$G_{{\circ}}b \in D$, for every $b \in D$.

We first prove the above claim in  case (B3).
Suppose that  $b \in D$.
Since, by assumption, $G_{{\circ}}x \leq G_{{\bullet}}x$, for every 
$x \in D$,
then  $ G_{{\circ}}G_{{\bullet}}b \leq G_{{\bullet}}G_{{\bullet}}b $,
by taking $x=G_{{\bullet}}b$ and 
since $ G_{{\bullet}}b \in D$.
Since  $G_{{\circ}}$ is assumed to be extensive, 
we get
 $G_{{\bullet}}b \leq G_{{\circ}}G_{{\bullet}}b \leq G_{{\bullet}}G_{{\bullet}}b \leq G_{{\bullet}}b$, where the last
inequality is obtained by
applying \eqref{isoide} to $G_{{\bullet}}$ with $G_{{\bullet}}b$
in place of $a$,
and using again the assumption that $G_{{\bullet}}b \in D$.

We have proved that $G_{{\bullet}}b  = G_{{\circ}}G_{{\bullet}}b  $, for all
 $b \in D$.
Now fix $x \in P$ and let $b$ vary in $D$. 
Whenever  $b$ is such that  $ x \leq G_{{\bullet}}b$,
then $x \leq  G_{{\bullet}}b =G_{{\circ}}G_{{\bullet}}b$  and, since  
$G_{{\bullet}}b \in D$,
we get $ \{ \,  G_{{\bullet}}b  \mid b\in D, x \leq G_{{\bullet}}b  \,\}
\subseteq 
\{ \,  G_{{\circ}}a  \mid a\in D, x \leq G_{{\circ}}a  \,\}$,
by considering $a= G_{{\bullet}}b$.
This implies  
$K_{{\circ}}x \leq K_{{\bullet}}x$.

We now consider case (B2).  We first compare
 $K_{{\circ}}$ and $H_{{\bullet}}$,
as defined by \eqref{kisoide1}, that is   
\begin{align}  
\labbel{kisoide1bulle}
   H_{{\bullet}}x &= 
\prod \{ \,  G_{{\bullet}}b  \mid b\in D \text{ and either }
  x \leq b \text{ or }  x \leq G_{{\bullet}}b  \,\}.
 \end{align}

Let us fix some $x \in P$. 
If, for some $b \in D$,  $ G_{{\bullet}}b$ belongs to the set in  
 \eqref{kisoide1bulle} because  $x \leq b$ holds,
then  $K_{{\circ}}x \leq K_{{\circ}}b =G_{{\circ}}b 
\leq  G_{{\bullet}}b$, since
$K_{{\circ}}$ is isotone, $K_{{\circ}}$ extends
$G_{{\circ}}$, 
and by the comparability assumption relating   
$G_{{\circ}} $ and $   G_{{\bullet}}$. 
 On the other hand, suppose that
 $ G_{{\bullet}}b$ belongs to the set in  
 \eqref{kisoide1bulle} because of $x \leq G_{{\bullet}}b$.
Since, by assumption, 
$ G_{{\bullet}}b \in D$,
then $K_{{\circ}}G_{{\bullet}}b=
G_{{\circ}}G_{{\bullet}}b$, hence 
$K_{{\circ}}x \leq K_{{\circ}}G_{{\bullet}}b=
G_{{\circ}}G_{{\bullet}}b \leq G_{{\bullet}}G_{{\bullet}}b
\leq G_{{\bullet}}b$, 
where, as in the case (B3), the last
inequality is obtained by
applying \eqref{isoide} to $G_{{\bullet}}$ with $G_{{\bullet}}b$
in place of $a$.   
Since we have showed $K_{{\circ}}x \leq  G_{{\bullet}}b$,
for every $b$ in the defining set for   $H_{{\bullet}}x$, we get
\begin{equation}\labbel{bago}    
K_{{\circ}}x \leq H_{{\bullet}}x.
   \end{equation}
Since $x$ was arbitrary in the above argument,
the inequality \eqref{bago}  holds for every $x \in P$. 

Since $K_{{\circ}}$ is isotone and idempotent,
then  $K_{{\circ}}x=K_{{\circ}}K_{{\circ}}x
 \leq K_{{\circ}}H_{{\bullet}}x
 \leq H_{{\bullet}}H_{{\bullet}}x$, for every $x \in P$, 
where we have applied
 \eqref{bago} twice, in the last inequality
 with $H_{{\bullet}}x$ in place of $x$.
Iterating, we get    
$K_{{\circ}}x \leq (K_{{\bullet}}) _{ \alpha } x$,
for every ordinal $\alpha$, 
where $(K_{{\bullet}}) _{ \alpha }$ denotes
the $\alpha$th stage of the construction of $K_{{\bullet}}$
from $H_{{\bullet}}$. Since 
$K_{{\bullet}} = (K_{{\bullet}}) _{ \alpha }$, for some $\alpha$, we get 
$K_{{\circ}}x \leq K_{{\bullet}} x$.
} 
\end{remark}

\arxiv{
\begin{remark} \labbel{necess} 
In the proof of Lemma \ref{exte} case (B2)
it is necessary to iterate  $H$.
Suppose that $b_1^1 \wedge b_2^1 =x$,
$b_1^1,  b_2^1 \in D$,  
$Gb_1^1<b_1^1 $, $ Gb_2^1 < b_2^1$
and  $Gb_1^1 \wedge  Gb_2^1 < x$.
According to \eqref{kisoide1}, we have
$Hx \leq Gb_1^1 \wedge  Gb_2^1$, and we might assume to be in the
situation in which   $Hx = Gb_1^1 \wedge  Gb_2^1$.
It might happen that 
$b_1^2 \wedge  b_2^2 =Hx$, 
$Gb_1^2<b_1^2 $, $ Gb_2^2 < b_2^2$
and  $Gb_1^2 \wedge  Gb_2^2 < Hx$.
for certain $b_1^2,  b_2^2 \in D$
incomparable with $b_1^1 $ and $  b_2^1 $.  
The above relations entail 
$HHx \leq Gb_1^2 \wedge  Gb_2^2 < Hx$.

The above construction can be iterated transfinitely in order to 
get examples in which $K^{ \alpha+1}x< K^ \alpha x $,
for an arbitrarily large ordinal $\alpha$. 
\end{remark}

\begin{remark} \labbel{isotinv}
In the situation
described in Lemma \ref{exte}
possible conditions for the existence of extensions
of an isotone involution  will necessarily be
much more involved. 
In fact, if $a \leq b$ and $K$ is an isotone involution,
then the order interval $[a,b]$
is isomorphic to  $[Ka,Kb]$.

A similar remark  applies to extensions of 
an antitone involution.
 \end{remark}

\begin{remark} \labbel{noc3}
The analogue of
Corollary \ref{plap} generally fails for case (C3). 
Let $\mathbf Q$ be the diamond with four elements
$0$, $1$, $a$, $b$, with  $a$ and $b$
not comparable.
Let $F$ be the binary function 
 $F(x,y) = x \wedge y$, which is isotone on
both components.
If  $\mathbf Q$ is obtained by  adding a new element $c$
with $0<c<a,b$, then the inclusion
is an order-embedding (not a lattice-embedding),
 $ x \wedge y \leq F(x,y) $ holds by construction in $\mathbf Q$,
but $a \wedge b =c >0=F(a,b)$ in $\mathbf P$.  
    
On the other hand, 
Corollary \ref{plap}(3-case of meets)(4)
apply also in  case (C3). 
 \end{remark}

 \begin{example} \labbel{apnofin}
(a) Let $T$ be the theory  of posets asserting
that, for every $n \in \mathbb N$,  
if there are less than $n$ elements, then the poset
is linearly ordered.
$T$ has the superamalgamation property,
the class of finite models of $T$ has the strong amalgamation property,
but not the superamalgamation property.

If $T^+$ extends $T$ in a language with two further
unary operations and asserts that the operations are isotone,
then $T^+$ has the superamalgamation property, but
the class of finite models of $T^+$ has not the amalgamation property.   

(b) The above theories are  not universal.
Let $T'$ be a theory in a language with 
two unary relation symbols $U$, $V$, and a binary
function $f$. The theory $T'$ asserts that     
if $U(x)$ and $V(y)$ hold, then all the elements   
$x$, $y$, $f(x,y)$,   $f(x,f(x,y))$, 
$f(x,f(x,f(x,y)))$\dots\  are distinct.
$T'$ has the strong amalgamation property,
but the class of finite models of $T'$ has not the  amalgamation property.

(c) Let $T''$ be as above, with a further
binary relation satisfying the axioms of posets.
Since there is no axiom connecting
$\leq$ with $U$, $V$ and $f$,
then $T''$ has the superamalgamation property.
$T''$ is universal and has JEP.
On the other hand, the class of finite models of $T''$
has neither JEP, nor AP; in particular,
it has not a Fra\"\i ss\'e limit.    
 \end{example}    
    
} 

 It is an open problem
whether the results in the present paper generalize 
to unary operations satisfying $K^n (x) = K x$, or
 $K^n (x) = x$, or, more generally,
$K^n (x) = K^m ( x)$, for some $m,n \in \mathbb N$.
Is it possible to consider more relationships connecting distinct
operations, other than comparability? 
For example,  do the results in the present paper generalize 
when adding two commuting unary operations, that is, satisfying
$K_{{\circ}} K_{{\bullet}} x = K_{{\bullet}}K_{{\circ}} x $?

\smallskip 

\emph{Acknowledgements.} We thank the referee for many useful
comments and for detecting some inaccuracies.

\section{Appendix. 
Superamalgamation into union} \labbel{supapu}

In this appendix  we deal with the case when   the superamalgamating structure can be taken
over the set-theoretical union of the base sets of the models to be amalgamated.
In this case no completion hypothesis is necessary and
the framework is slightly more general, to the effect that
we can work with a transitive binary relation, not
necessarily an order.

\smallskip 

There are  situations in which the completion hypothesis
(2) in Theorem \ref{superap} is not needed. 
First,  cases (A1e) - (A3)
are really elementary and do not need
the assumption (2), since no completeness assumption is necessary 
in the proof of Lemma \ref{exte} in such cases.

More interestingly, we do not need completions 
when the superamalgamating structure can be taken over the set 
$D= A \cup B$. 
In this situation we can work with arbitrary transitive relations in
place of orders, and we can also get a few additional results.
Notice that the definitions in Section \ref{prel},
in particular, the definition of superamalgamation,
apply to an arbitrary binary relation in place of $\leq$.

 The extension Lemma \ref{exte} is not needed
in the rest of the present section. 

\begin{definition} \labbel{supu}
Suppose that $\mathcal S$ is   a class of  
structures for the same language and with a specified  
binary relation $R$.
We say that $\mathcal S$ 
has the \emph{superamalgamation property into union} if,
under the assumptions in Definition \ref{sap}
(with $R$ in place of $\leq$), a superamalgamating 
structure exists over the set $D=A \cup B$.   
\end{definition}

Many examples of classes with  the   superamalgamation property into union
are presented in \cite{apu}, for example, the classes of models
with a binary relation, possibly satisfying some properties
chosen among transitivity, reflexivity, symmetry, antireflexivity, antisymmetry.
Other examples are classes with two binary 
transitive relations, one coarser than the other,
each satisfying some set of the above properties.

\begin{proposition} \labbel{solorel} 
Suppose that $\mathcal S$ is   a class of  
structures (or $T$ is a theory) with a specified  
binary relation $R$.
  \begin{enumerate}[(a)]   
 \item 
 If $\mathcal S$ is  
in  a language
without  operations of arity $\geq 2$ and  
$\mathcal S$ is closed under taking substructures,
then $\mathcal S$ satisfies the superamalgamation property  
into union
if and only if  $\mathcal S$ satisfies the superamalgamation property.
\item
If $T$ is a theory in some language $\mathscr L$, 
$\mathscr L' \supseteq \mathscr L$   
and $T$ has the superamalgamation property 
into union, then the class of models of $T$
in the language $\mathscr L'$    has the superamalgamation property 
into union.
\item
Suppose that  $\Sigma $ is a set of 
universal-existential sentences
in which at most one variable
 is bounded by the universal quantifier.
If $\mathcal S$ is a class of structures with
the superamalgamation property 
into union, then the class 
of all structures in $\mathcal S$ which satisfy 
$\Sigma$ has the superamalgamation property 
into union.
\item
Suppose that $( T_i) _{i \in I} $
is a sequence of theories in  languages $\mathscr L_i$,
and suppose that  $ \mathscr L_ i \cap \mathscr L _j = \{ R \} $,
for $i \neq j \in I$, 
where $R$ is a binary relation symbol.

If each $T_i$  has  the superamalgamation property 
into union
and asserts that $R$ is transitive, then $T=\bigcup_{i \in I} T_i$
has the superamalgamation property 
into union.
  \end{enumerate} 
\end{proposition} 

The proof of Proposition \ref{solorel}
is elementary, but the proposition is useful.
Cases (a) - (c) hold for the
strong amalgamation property, as well. 
See \cite{apu} for full details.
As far as (d) is concerned, notice that, since $R$ 
is assumed to be transitive, the superamalgamation
property determines the interpretation of $R$
on $A \cup B$. 

If $\mathbf A$ is a structure with a binary relation $R$,
we shall write $a \mathrel { R} b $ in place
of $R(a,b)$ or $(a,b) \in R$. A unary operation $K: A \to A$
is    \emph{$R$-isotone} (or \emph{$R$-preserving})
if $a \mathrel { R} b $ implies $Ka \mathrel { R} Kb $,
for every $a, b \in A$. The
operation $K$
is    \emph{$R$-antitone} (or \emph{$R$-reversing})
if $a \mathrel { R} b $ implies $Kb \mathrel { R} Ka $,
for every $a, b \in A$.

\begin{theorem} \labbel{super}
Suppose that $\mathcal S$ is   a class of  
structures with a  transitive 
binary relation $R$, and 
 $\mathcal S$ 
has the superamalgamation property into union.

If $\mathcal S_1$ is the class of expansions of structures 
of $\mathcal S$ obtained by adding
an $R$-isotone (an $R$-antitone) unary operation, 
then $\mathcal S_1$
has the superamalgamation property into union, in particular
the strong amalgamation property.

More generally, the same applies when  $\mathcal S_1$ is 
obtained by adding   families of such operations,
possibly adding  a set of comparability conditions.
Moreover, for every set 
$\Sigma $  of 
universal-existential sentences
in which at most one variable
 is bounded by the universal quantifier,
the class 
of all structures in $\mathcal S_1$ which satisfy 
$\Sigma$ has the superamalgamation property 
into union.
 \end{theorem}

 \begin{proof}
Given $\mathbf A$, $\mathbf B$, $\mathbf C$ 
as in Definition \ref{sap}, their $\mathcal S$-reducts
can be amalgamated into a structure 
$\mathbf D^-$ over $D=A \cup B$,
since the superamalgamation property is into union. 
Define $K$ on $D$ 
as in the proof of Theorem \ref{superap}, namely
\begin{equation}\labbel{kdd}
Kd =\begin{cases} 
K_{\mathbf A} d &    \text{if  $ d \in A  $},
\\ 
K_{\mathbf B} d &    \text{if  $ d \in B  $.}
 \end{cases} 
\end{equation}  

We only need to check that, when $\mathbf D^-$
is expanded by adding such a $K$, $R$-isotony is maintained.
Indeed, if $a, b \in A$ and
 $a \mathrel { R} b $, then
 $Ka \mathrel { R} Kb $, by $R$-isotony
on $\mathbf A$. The case when $a,b \in B$ 
 is similar.
Otherwise if, say, $a \in A \setminus B$ 
and $b \in B \setminus A$,
then, by superamalgamation, there exists
$c \in C= A \cap B$ such that $a \mathrel { R _{ \mathbf A} } c $
and $c \mathrel { R _{ \mathbf B} } b $.
By $R$-isotony of $K _{ \mathbf A}$ and $K _{ \mathbf B}$  
on $\mathbf A$ and $\mathbf B$,
we get
 $K _{ \mathbf A} a \mathrel { R _{ \mathbf A} } K _{ \mathbf A} c $
and $K _{ \mathbf B}c \mathrel { R _{ \mathbf B} } K _{ \mathbf B}b $,
that is,
 $Ka \mathrel { R _{ \mathbf D}} Kc $
and 
 $Kc \mathrel { R _{ \mathbf D}} Kb $,
according to  the definition \eqref{kdd}  of $K$, and since $\mathbf D$ 
extends $\mathbf D^-$ which amalgamates
the   $\mathcal S$-reducts of 
$\mathbf A$, $\mathbf B$ and $\mathbf C$. 
Since $\mathbf D^-$ belongs to 
 $\mathcal S$, then $R _{ \mathbf D}$ is transitive, thus we get
  $Ka \mathrel { R _{ \mathbf D}} Kb $.

The case of $R$-antitony is similar.

As for the last paragraph, and as in the proof of  
 Theorem \ref{superap},  we can add many new operations at a time,
since they do not influence each other. 
Clause \eqref{kdd} clearly preserves comparability 
conditions which already hold in $\mathbf A$ and $\mathbf B$.
The last statement follows
from Proposition \ref{solorel}(c). 
\end{proof}  

Though simple (and simply proved),
Proposition \ref{solorel}(c) is quite powerful.
For example, we can add conditions asserting that,
 $x \mathrel { R } Kx $, for every $x$, or, possibly,
 $Kx=KKx$, or $x=KKx$.
If we define recursively  $K^n$ 
by $K^0x=x$ and $K ^{n+1}x=KK^n x $,
we can add conditions of the form
  $K^nx \mathrel { R } K^mx $, or
$K^nx = K^mx $, for some fixed $m$ and $n$.
When dealing with more operations, all the universal closures 
of the following formulae can be taken in $\Sigma$    
 in \ref{solorel}(c): $KH x = HK x$, $K^mH^n x = H^nK^m x$,
 $K^1K^2 x \mathrel { R }  K_3 x$, 
$K^1K^2K_3 x \mathrel { R }  K^2K_4 x$, etc.

The assumption that $R$ is transitive is necessary in 
Theorem \ref{super}.
The  class of structures with a transitive relation $S$ and
a coarser binary relation $R$ has the superamalgamation property 
into union, with respect to $R$, but if a unary $R$-isotone operation is added,
 then the amalgamation property is lost. See \cite[Theorem 4.1(A)(C)]{apu}
in a slightly different terminology and
with the role of $R$ and $S$ exchanged.

\begin{corollary} \labbel{fr}
Suppose that $T$ is a universal theory in a finite relational language
 $\mathscr L$ containing a binary relation symbol $R$,
$T$ asserts that $R$ is transitive and $T$ has the superamalgamation property.

Let $T_1$ be the extension of $T$  in the language $\mathscr L \cup \{ K \} $
obtained by adding an axiom saying that
the unary operation  $K$ is 
idempotent (idempotent and isotone,
idempotent and extensive, a closure operation).
We can also add a finite number of such operations,
under the further assumption that they pairwise commute.

Then $T_1$ has model completion. 
\end{corollary} 

\begin{proof}
Since $T$ is universal in a relational language
and  $T$ has  the superamalgamation property,
then $T$ has the superamalgamation property into union, by 
Proposition \ref{solorel}(a). 
Then, by Theorem \ref{super}, $T_1$ has the amalgamation property.
Since $\mathscr L$ is relational and, in each case,
$K$ is idempotent
(and, if there are more operations, they 
pairwise commute), then $T_1$ is locally finite.
The conclusion follows from the well-known 
fact that every universal locally finite theory
with the amalgamation property in a finite language has model completion,
e.~g., \cite{Wh}. 
\end{proof}  

If $T_1$ in Corollary \ref{fr} has the joint embedding
property, we also get a Fra\"\i ss\'e model, arguing as in 
Corollary \ref{corsuperap} (3).   

The assumption that $\mathscr L$ is in a relational language in 
Corollary \ref{fr} can be somewhat relaxed; it is enough to
assume that  $T$ is locally finite with 
the superamalgamation property into union and 
that $T_1$ remains locally finite.

\renewcommand\refname{Additional References}

\end{document}